\documentclass[reqno]{amsart}

\usepackage{amscd}
\usepackage{amssymb}
\usepackage{amsmath}
\usepackage[dvips]{graphicx}
\usepackage{amsfonts}
\usepackage{amsopn}
\usepackage{amsthm}

\usepackage{amssymb}
\usepackage{graphicx}
\usepackage{mathrsfs}
\usepackage{hyperref}
\usepackage{amsrefs}
\usepackage{color}

\usepackage[normalem]{ulem}
\usepackage{accents}

\newcounter{mtheorem}

\newtheorem{theorem}{Theorem}[section]
\newtheorem{lemma}[theorem]{Lemma}
\newtheorem{prop}[theorem]{Proposition}

\theoremstyle{definition}
\newtheorem{definition}[theorem]{Definition}
\newtheorem{example}[theorem]{Example}

\theoremstyle{remark}
\newtheorem{remark}[theorem]{Remark}

\numberwithin{equation}{section}
\newcommand{\id}{{ \bf id}}

\newcommand{\ubar}[1]{\underaccent{\bar}{#1}} 

\def\({\begin{eqnarray}}
\def\){\end{eqnarray}}
\def\[{\begin{eqnarray*}}
\def\]{\end{eqnarray*}}

\def\endproof{\ \hfill QED. \bigskip}

\usepackage{amsaddr}

\title{An Optimal Transport Approach for the Kinetic Bohmian Equation}

\author{Wilfrid Gangbo}
\address{Department of Mathematics, University of California at Los Angeles, 
Los Angeles, CA 90095, U.S.A.}

\author{Jan Haskovec, Peter Markowich, Jesus Sierra}
\address{CEMSE Division, King Abdullah University of Science and Technology, Box 4700, Thuwal 23955-6900, Saudi Arabia}


\date{\today}

\begin{document}

\begin{abstract}
We study the existence theory of solutions of the kinetic Bohmian equation, a nonlinear Vlasov-type equation proposed for the phase-space formulation of Bohmian mechanics. Our main idea is to interpret the kinetic Bohmian equation as a Hamiltonian system defined on an appropriate Poisson manifold built on a Wasserstein space. We start by presenting an existence theory for stationary solutions of the kinetic Bohmian equation. Afterwards, we develop an approximative version of our Hamiltonian system in order to study its associated flow. We then prove existence of solutions of our approximative version. Finally, we present some convergence results for the approximative system, the aim being to establish that, in the limit, the approximative solution satisfies the kinetic Bohmian equation in a weak sense. 
\end{abstract}

\maketitle

\section{Introduction}

In this paper, we study the existence theory of solutions of the kinetic Bohmian equation \cite{markowich2010bohmian,markowich2012dynamics},
\begin{equation}
\partial_{t}\beta+v\cdot\nabla_{x}\beta-\nabla_{x}\left(V-\frac{1}{2}\frac{\triangle_{x}\sqrt{\varrho}}{\sqrt{\varrho}}\right)\cdot\nabla_{v}\beta=0,\label{eq:Vl_eqn}
\end{equation}
along with the initial value,
\begin{equation}
\beta\left(t=0,x,v\right)=\beta_{0}\in\mathcal{M}^{+}\left(\mathbb{R}^{d}\times\mathbb{R}^{d}\right),\label{eq:Vl_IC}
\end{equation}
where $v,x\in\mathbb{R}^{d}$, $t\geq0$, and $\mathcal{M}^{+}\left(\mathbb{R}^{d}\times\mathbb{R}^{d}\right)$
denotes the set of nonnegative Radon measures defined on $phase$
$space$, $\mathbb{R}^{d}\times\mathbb{R}^{d}$. Furthermore,
$V:\mathbb{R}^{d}\rightarrow\mathbb{R}$ is a potential satisfying some regularity assumptions given
below, and $\beta=\beta\left(t,x,v\right)$
represents the $generalized$ Bohmian measure. Finally, $\varrho=\varrho\left(t,x\right)$
is the position density given by
\[
\varrho\left(t,x\right)=\int_{\mathbb{R}^{d}}\beta\left(t,x,dv\right).
\]
For a comprehensive review of Bohmian mechanics and its
role in quantum mechanics, see, e.g., \cite{deurr2009bohmian,cushing2013bohmian}.

It was shown in \cite{markowich2010bohmian,markowich2012dynamics} that if the initial condition (\ref{eq:Vl_IC}) is a mono-kinetic measure, then there exists a connection between the kinetic Bohmian equation and the linear Schr\"{o}dinger equation that can be used to establish an existence theory for solutions of (\ref{eq:Vl_eqn}). Nevertheless, for the more general situation given by (\ref{eq:Vl_eqn})-(\ref{eq:Vl_IC}), such connection is lost. In this case, our analysis relies on interpreting the kinetic Bohmian equation as a Hamiltonian system on a space of probability measures in the following way. Let $\mathcal P_2(\mathbb R^d \times \mathbb R^d)$ stand for the set of Borel probability measures on $\mathbb R^d \times \mathbb R^d$ with finite second moments and consider the Hamiltonian $\mathcal H:\mathcal P_2(\mathbb R^d \times \mathbb R^d)\rightarrow\mathbb{R}\cup\{+\infty\} $ given by 
\[
\mathcal H(\mu):={1 \over 2} \int_{\mathbb R^d \times \mathbb R^d} |v|^2 \mu(dx, dv)+ \int_{\mathbb R^d \times \mathbb R^d} V(x) \mu(dx, dv)+  {1 \over 8} \int_{\mathbb R^d} {|\nabla \varrho|^2\over \varrho}\varrho(dx) + \chi_0\bigl((\pi^1_\# \mu)_s\bigr),
\]  
where we have used the Radon--Nikodym decomposition  
\[
\pi^1_\# \mu =\varrho \mathcal L^d + (\pi^1_\# \mu)_s,
\] 
$\pi^{1}:\mathbb{R}^{d}\times\mathbb{R}^{d}:\left(w,z\right)\rightarrow w$ represents the first projection of $\mathbb{R}^{d}\times\mathbb{R}^{d}$ onto $\mathbb{R}^{d}$, and $\chi_0: \mathcal P_2(\mathbb R^d) \rightarrow \{0, +\infty\}$ assumes the value $0$ on null measures and the value $+\infty$ on probability measures of positive total mass. Formally, at least, if the metric slope of $\mathcal H$ at $\mu$ is finite, under suitable conditions, the subdifferential  of $\mathcal H$ at $\mu$ is not empty. Its unique element of minimal norm is a Borel vector field, $\nabla_\mu \mathcal H: \mathbb R^d \times \mathbb R^d \rightarrow \mathbb R^d \times \mathbb R^d$, which is referred to as the Wasserstein gradient of $\mathcal H$ at $\mu$. $\nabla_\mu \mathcal H$ belongs to the range of the projection map 
\[
\pi_\mu: L^2(\mu) \rightarrow \overline{\nabla C_c^\infty(\mathbb R^d \times \mathbb R^d)}^{L^2(\mu)}
\]
and is given by  
\[
\nabla_\mu \mathcal H(x,v)= 
\left(
\begin{array}{c}
\nabla_x V(x) -{1 \over 2} \nabla_x \Bigl( {\triangle_x \sqrt \varrho \over \sqrt \varrho} \Bigr) (x) \\ \\
v\\
\end{array}
\right). 
\] 

Using the $(2d) \times (2d)$ symplectic matrix 
\[ J=\left( \begin{array}{cc}
0 & I_d  \\
-I_d & 0   \end{array} \right),\]
the theory developed in \cite{GangboKP} allows us to define a Poisson structure for which  $X_{ \mathcal H}:=\pi_\mu\bigl(J  \nabla_\mu \mathcal H\bigr)$ is a  Hamiltonian vector field; we have  
\[
X_{\mathcal H}(\mu)(x,v)= \pi_\mu \left(
\begin{array}{c}
v  \\ \\
-\nabla_x V(x) +{1 \over 2} \nabla_x \Bigl( {\triangle_x \sqrt \varrho \over \sqrt \varrho} \Bigr) (x)\\
\end{array}
\right).
\] 
On the other hand, the path $t \rightarrow \mu_t \in  \mathcal P_2(\mathbb R^d \times \mathbb R^d)$ is said to be driven by a velocity vector field, ${\mathbf v}: (0,1) \times \mathbb R^d \times \mathbb R^d \rightarrow \mathbb R^d \times \mathbb R^d$, if 
\[
\partial_t \mu + \nabla \cdot (\mu {\mathbf v})=0,
\] 
in the sense of distributions. According to \cite{ambrosio2008hamiltonian}, the path $t \rightarrow \mu_t$ satisfies the Hamiltonian system (defined in the context of Poisson geometry) 
\[
\dot \mu= X_{\mathcal H}(\mu) 
\]
if $X_{\mathcal H}(\mu)$ is a velocity vector field driving $t \rightarrow \mu_t$, namely,
\[
\partial_t \mu+ \nabla_x \cdot (v \mu)= \nabla_v \cdot \Bigl[ \mu \Bigl( \nabla_x V(x) -{1 \over 2} \nabla_x \bigl( {\triangle_x \sqrt \varrho \over \sqrt \varrho} \bigr) \Bigr) \Bigr],
\] 
in the sense of distributions. This is exactly (\ref{eq:Vl_eqn}) when $\mu_t = \beta(t, \cdot, \cdot) \mathcal L^{2d}.$

Therefore, now we can say that one of the main ideas of this paper is to investigate the existence theory of solutions of the kinetic Bohmian equation through the Hamiltonian flow generated by the Hamiltonian vector field $X_{ \mathcal H}$.

To motivate the study of the kinetic Bohmian equation, let us start by reviewing the aforementioned connection with the linear Schr\"{o}dinger equation,  

\begin{equation}
i\partial_{t}\psi=-\frac{1}{2}\triangle\psi+V\psi,\quad\psi\left(t=0, \cdot \right)=\psi_{0}\in L^{2}\left(\mathbb{R}^{d};\mathbb{C}\right).\label{eq:Sch}
\end{equation}
A thorough analysis of this equation can be found in, e.g., \cite{cazenave2003semilinear,sulem1999nonlinear,tao2006nonlinear}. We adopt the normalization of the initial data, i.e., $\left\Vert \psi_{0}\right\Vert _{L^{2}}=1$. Thus,

\begin{equation}
\left\Vert \psi\left(t\right)\right\Vert _{L^{2}}=\left\Vert \psi_{0}\right\Vert _{L^{2}}=1.\label{eq:mass}
\end{equation}
In addition, we assume that $\psi$ has bounded
initial energy. The energy is conserved for all $t\geq0$ and is given
by

\[
E\left(t\right):=\frac{1}{2}\int_{\mathbb{R}^{d}}\left|\nabla\psi\left(t,x\right)\right|^{2}dx+\int_{\mathbb{R}^{d}}V\left(x\right)\left|\psi\left(t,x\right)\right|^{2}dx=E\left(0\right).
\]
Note that the Schr\"{o}dinger equation (\ref{eq:Sch}) has a reduced Planck constant equal to one ($\hbar = 1$). 

As a consequence of (\ref{eq:mass}), one can define real-valued probability
densities from $\psi\left(t,x\right)\in\mathbb{C}$.
These probability densities can be used to compute expectation values of physical observables. In particular, we have the $position$ and $current$ $densities$ given by

\begin{equation}
\varrho=\varrho\left(t,x\right)=\left|\psi\left(t,x\right)\right|^{2},\quad J =J\left(t,x\right)=\mathrm{Im}\left(\overline{\psi}\left(t,x\right)\nabla\psi\left(t,x\right)\right).\label{eq:dens_curr}
\end{equation}

\begin{definition}
(Bohmian measure \cite{markowich2010bohmian,markowich2012dynamics}). For $\psi\in H^{1}\left(\mathbb{R}^{d}\right)$,
with associated densities $\varrho$, $J$
given by (\ref{eq:dens_curr}), the Bohmian measure $\beta=\beta\left[\psi\right]\in\mathcal{M}^{+}\left(\mathbb{R}^{d}\times\mathbb{R}^{d}\right)$ is defined by

\begin{equation}
\left\langle \beta,\varphi\right\rangle :=\int_{\mathbb{R}^{d}}\varrho\left(x\right)\varphi\left(x,\frac{J \left(x\right)}{\varrho\left(x\right)}\right)dx,\quad\forall\varphi\in C_{0}\left(\mathbb{R}^{d}\times\mathbb{R}^{d}\right),\label{eq:Bohm_me}
\end{equation}
 where $C_{0}\left(\mathbb{R}^{d}\times\mathbb{R}^{d}\right)$
denotes the space of continuous functions vanishing at infinity.
\end{definition}
Let
\begin{equation}
\beta_{0}=\beta_{0}\left(x,v\right)=\varrho_{0}\left(x\right)\delta\left(v-u_{0}\left(x\right)\right),\label{eq:initial_data}
\end{equation}
where $\varrho_{0}\equiv\varrho\left(t=0,x\right)$, $u_{0}\equiv u\left(t=0,x\right)$,  $u = u\left(t,x\right):=J / \varrho$, and $\delta$ is the delta distribution on $\mathbb{R}^{d}$.  It was shown in \cite{markowich2010bohmian} that if $\psi\left(t,x\right)$
solves the Schr\"{o}dinger equation (\ref{eq:Sch}), then the corresponding
Bohmian measure (\ref{eq:Bohm_me}) is the push-forward of (\ref{eq:initial_data})
under the phase space flow

\[
\Phi_{t}:\left(x,v\right)\mapsto\left(X\left(t,x,v\right),P\left(t,x,v\right)\right),
\]
induced by

\begin{equation}
\begin{cases}
\begin{array}{l}
\dot{X}=P,\\
\dot{P}=-\nabla V\left(X\right)-\nabla V_{B}\left(t,X\right),
\end{array}\end{cases}\label{eq:B_flow}
\end{equation}
where $V_{B}\left(t,x\right)$ is the Bohm potential:
\[
V_{B}\left(t,x\right):=-\frac{1}{2}\frac{\triangle\sqrt{\varrho\left(t,x\right)}}{\sqrt{\varrho\left(t,x\right)}}.
\]
Note that the specific form of the initial data (\ref{eq:initial_data})
implies that the phase-space flow $\Phi_{t}$, governed
by (\ref{eq:B_flow}), is initially projected onto the graph of $u_{0}$,
that is,
\begin{equation}
\mathcal{L}:=\left\{ \left(x,v\right)\in\mathbb{R}^{d}\times\mathbb{R}^{d}:v=u_{0}\left(x\right)\right\} .\label{eq:lag_man} 
\end{equation}
This imposes a big limitation for the application of the theory developed in \cite{markowich2010bohmian,markowich2012dynamics}: from the whole
phase space, we are restricted to the $Lagrangian$ $submanifold$
(\ref{eq:lag_man}) for the initial condition of (\ref{eq:B_flow}).

Furthermore, it was proved in \cite{markowich2012dynamics} that for
$V\in C_{b}^{1}\left(\mathbb{R}^{d};\mathbb{R}\right)$ and $\psi_{0}\in H^{3}\left(\mathbb{R}^{d}\right)$
with corresponding $\varrho_{0}$, $J_{0}$
given by (\ref{eq:dens_curr}), the Bohmian
measure
\[
\beta\left(t,x,v\right)=\varrho\left(t,x\right)\delta\left(v-u\left(t,x\right)\right),
\]
is a weak solution of the kinetic Bohmian equation in $\mathcal{D}'\left(\mathbb{R}\times\mathbb{R}^{d}\times\mathbb{R}^{d}\right)$
and in $\mathcal{D}'\left(\left[0,\infty\right)\times\mathbb{R}^{d}\times\mathbb{R}^{d}\right)$
with initial data (\ref{eq:initial_data}). On the other hand, the
uniqueness theory is still an open problem.

As mentioned before, the purpose of this paper is to study the kinetic Bohmian equation
with the more general initial data (\ref{eq:Vl_IC}), which implies that the connection with the Schr\"{o}dinger equation is lost. Nevertheless, the
idea is to use the Wasserstein gradient/Hamiltonian flow techniques
to generate rigorous results on (\ref{eq:Vl_eqn})-(\ref{eq:Vl_IC})
with the aim of overcoming the limitations mentioned above, in particular, the restriction from the whole phase space to the Lagrangian submanifold (\ref{eq:lag_man}). Moreover, this opens the door for a new interpretation of Bohmian mechanics through optimal transportation.

The remainder of this paper is organized as follows. In Section 2, we presents the basic theory and notation used throughout our analysis. In Section 3, we study the existence of stationary solutions of the kinetic Bohmian equation. Sections 4, 5, and 6 are devoted to the development of an approximative version of the kinetic Bohmian equation; in particular, we prove existence of solutions of this approximative version in Section 6. In Section 7, we present some convergence results for the approximative model developed in Sections 4, 5, and 6. Conclusions are drawn in Section 8.

%
%
%

\section{Preliminaries}\label{sec:prelim}
Since most of our work is performed inside the framework of
probability measures, we present now the basic concepts and notation
for this topic. A comprehensive review of this subject can be found in \cite{parthasarathy1967probability}. Furthermore, the theory of optimal transportation is extensively studied in \cite{ambrosio2008gradient,villani2003topics,santambrogio2015optimal}.

A Borel measure on a topological space, $X$, is any measure defined on the $\sigma-$algebra generated by the open sets of $X$. The elements of such $\sigma-$algebra are called the Borel sets. Furthermore, a map, $f:X \rightarrow Y$, between the topological spaces $X$ and $Y$, is called a Borel map if $f^{-1}(B)$ is a Borel set for any Borel set $B \subset Y$

Suppose that $\mu$ and $\nu$ are nonnegative Borel measures on the
topological spaces $X$ and $Y$, respectively. We say that the Borel
map $T:X\rightarrow Y$ transports $\mu$ into $\nu$, denoted by
$T_{\#}\mu=\nu$, if for every Borel set $B\subset Y$ we have $\nu\left[B\right]=\mu\left[T^{-1}\left(B\right)\right]$; in this case, we also say the $\nu$ is the pushforward of $\mu$ through $T$. We shall represent by $\mathcal{J}\left(\mu,\nu\right)$ the set of
all Borel maps, $T$, satisfying $T_{\#}\mu=\nu$.

Let $\pi^{1}:X \times Y \rightarrow X$ be the projection of $X \times Y$ onto $X$ and let $\pi^{2}:X \times Y \rightarrow Y$ be the projection of $X \times Y$ onto $Y$. A nonnegative Borel measure, $\gamma$, on $X\times Y$ is said to have marginals
$\mu$ and $\nu$ if $\mu=\pi^{1}_{\#}\gamma$ and $\nu=\pi^{2}_{\#}\gamma$;
in this case, $\gamma$ is called a transport plan between $\mu$
and $\nu$. The set of all transport plans between $\mu$ and $\nu$
is denoted by $\Gamma\left(\mu,\nu\right)$.

Let $d\geq1$ be an integer and let $D\in\left\{ d,2d\right\} $.
The $D$-dimensional Lebesgue measure on $\mathbb{R}^{D}$ is represented
by $\mathcal{L}^{D}$. $\mathcal{P}\left(\mathbb{R}^{D}\right)$ stands
for the set of Borel probability measures on $\mathbb{R}^{D}$. The
second moment of $\mu\in\mathcal{P}\left(\mathbb{R}^{D}\right)$ is
defined as
\[
M_{2}\left(\mu\right):=\int_{\mathbb{R}^{D}}\left|z\right|^{2}d\mu\left(z\right).
\]
Furthermore,
\[
\mathcal{P}_{2}\left(\mathbb{R}^{D}\right):=\left\{ \mu\in\mathcal{P}\left(\mathbb{R}^{D}\right):M_{2}\left(\mu\right)<+\infty\right\} .
\]
The subspace of $\mathcal{P}_{2}\left(\mathbb{R}^{D}\right)$ of absolutely
continuous measures with respect to $\mathcal{L}^{D}$ is represented
by $\mathcal{P}_{2}^{r}\left(\mathbb{R}^{D}\right)$. 

For $\mu\in\mathcal{P}_{2}\left(\mathbb{R}^{D}\right)$, we denote
by $L^{2}\left(\mu\right)$ the set of Borel vector fields, $\xi:\mathbb{R}^{D}\rightarrow\mathbb{R}^{D}$,
which are $\mu-$measurable and satisfy
\[
\left\Vert \xi\right\Vert _{\mu}^{2}:=\int_{\mathbb{R}^{D}}\left|\xi\left(z\right)\right|^{2}d\mu\left(z\right)<+\infty.
\]

$\mathcal{P}_{2}\left(\mathbb{R}^{D}\right)$ is canonically endowed
with the Wasserstein distance, $W_{2}$, defined by
\begin{equation}
W_{2}^{2}\left(\mu,\nu\right):=\underset{\gamma}{\min}\left\{ \int_{\mathbb{R}^{D}\times\mathbb{R}^{D}}\left|x-y\right|^{2}d\gamma\left(x,y\right):\gamma\in\Gamma\left(\mu,\nu\right)\right\} .\label{eq:W_dist}
\end{equation}
Any minimizer in (\ref{eq:W_dist}) is called an optimal transport
plan between $\mu$ and $\nu$. The set of all such minimizers is
indicated by $\Gamma_{o}\left(\mu,\nu\right)$.

Suppose now that $\mu\in\mathcal{P}_{2}^{r}\left(\mathbb{R}^{D}\right)$
and $\nu\in\mathcal{P}_{2}\left(\mathbb{R}^{D}\right)$. Then, there
exists a unique minimizer, $\gamma_{o}$, in (\ref{eq:W_dist}) which
can be represented as $\gamma_{o}=\left(\mathbf{id}\times T_{\mu}^{v}\right)_{\#}\mu$
for some $T_{\mu}^{\nu}:\mathbb{R}^{D}\rightarrow\mathbb{R}^{D}$
that coincides $\mu-$a.e. with the gradient of a convex function and satisfies $T_{\mu\#}^{\nu}\mu=\nu$.
Hence, $T_{\mu}^{\nu}$ is the unique minimizer of 
\[
T\rightarrow\int_{\mathbb{R}^{D}}\left|z-T\left(z\right)\right|^{2}d\mu\left(z\right),
\]
over $\mathcal{J}\left(\mu,\nu\right)$.

$\left(\mathcal{P}_{2}\left(\mathbb{R}^{D}\right),W_{2}\right)$ is
a Polish space, namely, a complete and separable metric space (see Section
7.1 in \cite{villani2003topics} and Proposition 7.1.5 in \cite{ambrosio2008gradient}
for details). On the other hand, it is not locally compact. Nevertheless,
bounded sets in $\mathcal{P}_{2}\left(\mathbb{R}^{D}\right)$ are
sequentially relatively compact with respect to the narrow convergence; a sequence $\left( \mu_{k}\right) _{k}\subset\mathcal{P}\left(\mathbb{R}^{D}\right)$
converges narrowly to $\mu\in\mathcal{P}\left(\mathbb{R}^{D}\right)$
as $k\rightarrow\infty$ if
\[
\underset{k\rightarrow\infty}{\lim}\int_{\mathbb{R}^{D}}g\left(z\right)d\mu_{k}\left(z\right)=\int_{\mathbb{R}^{D}}g\left(z\right)d\mu\left(z\right),
\]
for every $g\in C_{b}^{0}\left(\mathbb{R}^{D}\right)$, the space
of bounded and continuous functions on $\mathbb{R}^{D}$. Moreover,
a sequence $\left( \beta_{k}\right) _{k}\subset\mathcal{P}_{2}\left(\mathbb{R}^{D}\right)$
converges to $\beta\in\mathcal{P}_{2}\left(\mathbb{R}^{D}\right)$
if and only if $\left( \beta_{k}\right) _{k}$ converges
narrowly to $\beta$ and $M_{2}\left(\beta_{k}\right)\rightarrow M_{2}\left(\beta\right)$
as $k\rightarrow\infty$.

A particularly important subject for our analysis is the differentiable
Riemannian structure of $\mathcal{P}_{2}\left(\mathbb{R}^{D}\right)$,
which can be derived from its metric structure. For such derivation,
we first have to characterize the absolutely continuous curves $\mu_{t}:\left[a,b\right]\rightarrow\mathcal{P}_{2}\left(\mathbb{R}^{D}\right)$.
As proved in Theorem 8.3.1 of \cite{ambrosio2008gradient}, if $\mu_{t}$
solves the continuity equation
\begin{equation}
\frac{d}{dt}\mu_{t}+\nabla\cdot\left(w_{t}\mu_{t}\right)=0,\label{eq:cont_eq}
\end{equation}
in the sense of distributions in $\left(a,b\right)\times\mathbb{R}^{D}$
for some time-dependent velocity vector field, $w_{t}$, with $\left\Vert w_{t}\right\Vert _{\mu_{t}}\in L^{1}\left(a,b\right)$,
then
\begin{equation}
W_{2}\left(\mu_{s},\mu_{t}\right)\leq\int_{s}^{t}\left\Vert w_{\tau}\right\Vert _{\mu_{\tau}}d\tau\qquad\forall a\leq s<t\leq b.\label{eq:AB}
\end{equation}
Therefore, the map $t\mapsto\mu_{t}$ is absolutely continuous from
$\left[a,b\right]$ to $\mathcal{P}_{2}\left(\mathbb{R}^{D}\right)$.
Conversely, for any absolutely continuous curve, $t\mapsto\mu_{t}$,
there exists a unique (up to $\mathcal{L}^{1}-$negligible sets in
time) velocity vector field, $v_{t}$, for which the continuity equation
(\ref{eq:cont_eq}) holds, along with asymptotic equality in (\ref{eq:AB}):
\[
\underset{h\rightarrow0}{\lim}\frac{1}{\left|h\right|}W_{2}\left(\mu_{t+h},\mu_{t}\right)=\left\Vert v_{t}\right\Vert _{\mu_{t}}\quad\textrm{for a.e. }t.
\]
Proposition 8.4.5 of \cite{ambrosio2008gradient} shows that this
minimality property of $v_{t}$ is equivalent to the fact that
\[
v_{t}\in\overline{\left\{ \nabla\varphi:\varphi\in C_{c}^{\infty}\left(\mathbb{R}^{D}\right)\right\} }^{L^{2}\left(\mu_{t}\right)}.
\]
This result leads to the identification of $v_{t}$ as the ``tangent''
velocity vector to $\mu_{t}$. Hence, the tangent space to $\mathcal{P}_{2}\left(\mathbb{R}^{D}\right)$
at $\mu$ is defined as
\[
T_{\mu}\mathcal{P}_{2}\left(\mathbb{R}^{D}\right):=\overline{\left\{ \nabla\varphi:\varphi\in C_{c}^{\infty}\left(\mathbb{R}^{D}\right)\right\} }^{L^{2}\left(\mu\right)}.
\]
Furthermore, using a simple duality argument, it has been proved in
Lemma 8.4.2 of \cite{ambrosio2008gradient} that
\[
\left[T_{\mu}\mathcal{P}_{2}\left(\mathbb{R}^{D}\right)\right]^{\perp}=\left\{ w\in L^{2}\left(\mu\right):\nabla\cdot\left(w\mu\right)=0\right\} .
\]

The following is a useful characterization of the tangent velocity
vector, $v_{t}$, given in Proposition 8.4.6 of \cite{ambrosio2008gradient}:
\[
\underset{h\rightarrow0}{\lim}\left(w,\frac{z-w}{h}\right)_{\#}\gamma_{h}=\left(\mathbf{id},v_{t}\right)_{\#}\mu_{t}\quad\textrm{in }\mathcal{P}_{2}\left(\mathbb{R}^{D}\times\mathbb{R}^{D}\right),
\]
for almost every $t$ and any $\gamma_{h}\in\Gamma_{o}\left(\mu_{t},\mu_{t+h}\right)$.
In addition, if $\mu_{t}\in\mathcal{P}_{2}^{r}\left(\mathbb{R}^{D}\right)$,
then the last characterization becomes
\[
\frac{t_{h}-\mathbf{id}}{h}\rightarrow v_{t}\quad\textrm{in }L^{2}\left(\mu_{t};\mathbb{R}^{D}\right)\textrm{ as }h\rightarrow0,
\]
where $t_{h}$ are the optimal transport maps between $\mu_{t}$ and
$\mu_{t+h}$.

We present now some basic results from convex analysis in $\mathcal{P}_{2}\left(\mathbb{R}^{D}\right)$
which are extensively used in the sequel.

Let $\mu_{0},\mu_{1}\in\mathcal{P}_{2}\left(\mathbb{R}^{D}\right)$
and let $\gamma\in\Gamma_{o}\left(\mu_{o},\mu_{1}\right)$. Let $\pi^{1}:\mathbb{R}^{D}\times\mathbb{R}^{D}:\left(w,z\right)\rightarrow w$
and $\pi^{2}:\mathbb{R}^{D}\times\mathbb{R}^{D}:\left(w,z\right)\rightarrow z$
be the first and second projections of $\mathbb{R}^{D}\times\mathbb{R}^{D}$
onto $\mathbb{R}^{D}$, respectively. Consider the interpolation between
the measures $\mu_{0}$ and $\mu_{1}$ given by
\[
\mu_{t}=\left(\left(1-t\right)\pi^{1}+t\pi^{2}\right)_{\#}\gamma,\quad t\in\left[0,1\right].
\]
Theorem 7.2.2 of \cite{ambrosio2008gradient} shows that $t\mapsto\mu_{t}$
is a constant speed geodesic in $\mathcal{P}_{2}\left(\mathbb{R}^{D}\right)$,
i.e., $W_{2}\left(\mu_{s},\mu_{t}\right)=\left|t-s\right|W_{2}\left(\mu_{0},\mu_{1}\right)$
for all $s,t\in\left[0,1\right]$. In addition, any constant speed
geodesic has this representation for a suitable optimal transport
plan, $\gamma$. 

Let $\phi:\mathcal{P}_{2}\left(\mathbb{R}^{D}\right)\rightarrow\left[-\infty,+\infty\right]$.
We define the effective domain of $\phi$ as
\[
D\left(\phi\right):=\left\{ z\in\mathcal{P}_{2}\left(\mathbb{R}^{D}\right):-\infty<\phi\left(z\right)<+\infty\right\} .
\]

\begin{definition}
{($\lambda-$convexity)}. Let $\phi:\mathcal{P}_{2}\left(\mathbb{R}^{D}\right)\rightarrow\left[-\infty,+\infty\right]$
be such that $D\left(\phi\right)\neq\emptyset$ and let $\lambda\in\mathbb{R}$.
We say that $\phi$ is $\lambda-$convex if for every $\mu_{0},\mu_{1}\in\mathcal{P}_{2}\left(\mathbb{R}^{D}\right)$
and every $\gamma\in\Gamma_{o}\left(\mu_{0},\mu_{1}\right)$ we have
\[
\phi\left(\mu_{t}\right)\leq\left(1-t\right)\phi\left(\mu_{0}\right)+t\phi\left(\mu_{1}\right)-\frac{\lambda}{2}t\left(1-t\right)W_{2}^{2}\left(\mu_{0},\mu_{1}\right)\quad\forall t\in\left[0,1\right],
\]
where $\mu_{t}=\left(\left(1-t\right)\pi^{1}+t\pi^{2}\right)_{\#}\gamma$.
In particular, $0-$convexity corresponds to the so-called displacement
convexity. 
\end{definition}

\begin{definition}\label{de:differentiability1} Let $\mathcal G: \mathcal P_2(\mathbb R^D) \rightarrow [-\infty, \infty]$ be such that $D(\mathcal G)\not =\emptyset$ and let  $\mu \in D(\mathcal G)$. 
\begin{enumerate} 
\item[(i)] We say that $\xi$ belongs to the subdifferential of $\mathcal G$ at $\mu$, and we write $\xi \in \ubar{\partial} \mathcal G$, if $\xi \in L^2(\mu)$ and 
\begin{equation}\label{eq:differentiability1} 
\mathcal G(\nu)-\mathcal G(\mu)\geq \sup_{\gamma\in\Gamma_{o}(\mu, \nu) } \int_{\mathbb R^D \times \mathbb R^D} \xi( w) \cdot (z- w) \gamma( dw,dz )+o\big(W_2(\mu,\nu)\big) \qquad \forall \; \nu\in D(\mathcal G).
\end{equation} 
The unique element of minimal norm in $ \ubar{\partial} \mathcal G(\mu)$ belongs to $T_\mu \mathcal P_2(\mathbb R^D)$ and is called the gradient of $\mathcal G$ at $\mu$; it is denoted by $\nabla_\mu \mathcal G(\mu).$  
\item[(ii)] We say that $\xi$ belongs to the superdifferential of $\mathcal G$ at $\mu$, and we write $\xi \in \bar{\partial} \mathcal G(\mu)$, if $-\xi \in \ubar{\partial} (-\mathcal G)(\mu).$ 
\item[(iii)]  We say that $\mathcal G$ is differentiable at $\mu$ if both $\ubar{\partial} \mathcal G(\mu)$ and $\bar{\partial} \mathcal G(\mu)$ are non empty. In that case (see e.g.  \cite{GangboNT2}) both sets coincide and 
\[ 
\ubar{\partial} \mathcal G (\mu)\cap T_\mu \mathcal P_2(\mathcal R^D) =\bar{\partial} \mathcal G(\mu)\cap T_\mu \mathcal P_2(\mathbb R^D)=\{\nabla_\mu \mathcal G(\mu)\}.
\]
\end{enumerate} 
Therefore, there is no ambiguity if we define the gradient of $\mathcal G$ at $\mu$  as the unique element of minimal norm in $ \bar{\partial} \mathcal G(\mu)$; we denote it by $\nabla_\mu \mathcal G(\mu).$
\end{definition} 

\begin{remark}\label{re:differentiability1} Here are some remarks.
\begin{enumerate} 
\item[(i)] We refer the reader to Remark 3.2 of \cite{GangboNT2} for property (iii) in Definition \ref{de:differentiability1}.
\item[(ii)]  Due to Proposition 8.5.4 of \cite{ambrosio2008gradient}, (\ref{eq:differentiability1}) holds for $\xi$ if and only if it holds for any $\xi_0 \in L^2(\mu)$ such that $\xi_0 -\xi$ belongs to the orthogonal complement of $T_\mu \mathcal P_2(\mathbb R^D)$ in $L^2(\mu).$  Rephrasing, if  (\ref{eq:differentiability1}) holds for $\xi_0 \in L^2(\mu)$, then it holds for $\xi$ defined as the orthogonal projection of $\xi_0$ onto $T_\mu \mathcal P_2(\mathbb R^D).$ Hence, 
\[\nabla_\mu \Phi(\mu) + \{\xi \in L^2(\mu)\; | \; {\rm div\,}_\mu(\xi)=0 \} \subset  \bar{\partial} \Phi(\mu).\]
\item[(iii)]  Define $\psi(\nu)=1 / 2W_2^2(\nu, \varrho)$ for $\nu \in \mathcal P_2(\mathbb R^D),$ where $\varrho \in \mathcal P_2(\mathbb R^D)$ is absolutely continuous. The proof of Proposition 10.4.12 \cite{ambrosio2008gradient} reveals that if $\xi \in \bar{\partial} \psi(\nu)$, since $\gamma \in \Gamma_{o}(\nu, \varrho)$ has a unique element, then  $\pi_\nu(\xi)= {\bf id}-\bar \gamma$, where $\bar \gamma$ is the barycentric projection of $\gamma$. Hence, 
\[
\bar{\partial} \psi(\nu)= {\bf id}-\bar \gamma + \bigl\{v \in L^2(\nu)\; | \;  {\rm div\,}_\nu(v)=0 \bigr\}. 
\]
\end{enumerate}
\end{remark}

We next list some facts about proper functionals, $\Phi: \mathcal P_2(\mathbb R^d  \times \mathbb R^d) \rightarrow \mathbb R \cup \{\infty\}$, for which there exists  a functional, $\phi: \mathcal P_2(\mathbb R^d) \rightarrow \mathbb R \cup \{\infty\}$, such that 
\[
\Phi(\mu)=\phi(\pi^1_\# \mu).
\] 
If $\xi= (\xi_1, \xi_2) \in {\ubar \partial    \Phi(\mu)}$,  then ${\bar \xi}_1 \in {\ubar \partial} \phi(\varrho)$,  where 
\[ 
{\bar \xi}_1(x)= \int_{\mathbb R^d} \xi_1(x, v) \mu_x(dv)
\]
and $(\mu_x)_{x \in \mathbb R^d}$ is the disintegration of $\mu$ with respect to $\varrho.$ This result holds under the assumption that $\bar \partial \Phi(\mu) \not =\emptyset.$ Moreover,  if $\Phi$ is bounded below and lower semicontinuity for the narrow convergence, we can then draw some conclusions about the functionals $\Phi_\tau$ defined in  \ref{eq:dec20.2015.1}, the Moreau--Yosida approximations of $\Phi$. First,  $\bar \partial \Phi_\tau(\mu) \not =\emptyset$ and 
\[
\Phi_\tau(\mu)= \phi_\tau(\varrho).
\]   
Second, if we further assume that the domain of $\phi$ is contained in $\mathcal P_2^r(\mathbb R^d)$ and $\varrho \in \mathcal P_2^r(\mathbb R^d)$, then $\bar \partial \Phi_{\tau}(\mu)$ and $\bar \partial \phi_{\tau}(\varrho)$ are non empty and their elements of minimal norm, respectively denoted by  $\nabla_\mu \Phi_{\tau}(\mu)$ and $\nabla_\varrho \phi_{\tau}(\varrho)$, satisfy 
\[
\nabla_\mu \Phi_\tau(\mu)(x,v)=  
\left(
\begin{array}{c}
\nabla_\varrho \phi_\tau(\varrho)(x) \\ \\
0\\
\end{array}
\right). 
\] 
This is a subtle statement, since  (cf. Remark \ref{re:differentiability1} (ii))
 \[
\nabla_\varrho \phi(\varrho) + \{u \in L^2(\varrho)\; | \; {\rm div\,}_\varrho(u)=0 \} \subset   \bar{\partial} \phi(\varrho) 
 \] 
 and similarly, 
\begin{equation}\label{eq:jan07.2016.1}
\nabla_\mu \Phi(\mu) + \{\xi \in L^2(\mu)\; | \; {\rm div\,}_\mu(\xi)=0 \} \subset  \bar{\partial} \Phi(\mu).
 \end{equation} 
Thus, there are elements, $\Sigma$, of $\ubar{\partial} \Phi(\mu)$ which are functions of $(x,v)$ and have second components that are not null. To see this, it suffices to choose $\xi$ such that $ {\rm div\,}_\mu(\xi)=0$ with $\xi(x,v)$ depending on $(x,v)$ and $\pi^2(\xi) \not =0$; then, just set $\Sigma= \nabla_\mu \Phi(\mu)+\xi$. 

Finally, for simplicity of notation, we define the Fisher information, $8\mathcal F$, by (see \cite{gianazza2009wasserstein,matthes2009family}):  
\begin{equation}\label{eq:Fisher}
8 \mathcal F(\varrho):= 
\left\{
\begin{array}{l}
4\int_{\mathbb{R}^{D}} |\nabla \sqrt {\varrho}|^2 dx\quad \quad \hbox{if} \quad \sqrt \varrho \in W^{1,2}(\mathbb{R}^{D}) \cap\{\varrho \geq 0\}, \\
\hfill +\infty \qquad \qquad \quad \, \hbox{if} \quad \sqrt \varrho \not \in W^{1,2}(\mathbb{R}^{D}) \cap\{\varrho \geq 0\}.
\end{array}
\right.
\end{equation}
The Fisher information plays a fundamental role in our subsequent analysis.

%
%
%
\section{Stationary solutions on the tangent bundle $TM:=\mathbb R^d \times \mathbb R^d$}
In this section, we start our analysis by exploring special solutions of the kinetic Bohmian equation (\ref{eq:Vl_eqn}). To this end, define the Hamiltonian function
\[
H\left(x,v\right):=\frac{1}{2}\left|v\right|^{2}-\frac{1}{2}\frac{\triangle\sqrt{\varrho\left(x\right)}}{\sqrt{\varrho\left(x\right)}}+V\left(x\right),
\]
and consider solutions of (\ref{eq:Vl_eqn}) of the form
\[
\beta\left(x,v\right)=F\left(H\left(x,v\right)-\eta\right),
\]
where $\eta\in\mathbb{R}$ represents a (quasi) Fermi level and $F:\mathbb{R}\rightarrow\mathbb{R}^{+}$ is a continuous strictly decreasing function. In particular, we are interested in functions $F:\mathbb{R}\rightarrow\mathbb{R}^{+}$ satisfying 
\begin{equation}\label{eq:defn-of-A}
A\left(\alpha\right):=\int_{\mathbb{R}^{D}}F\left(\frac{1}{2}\left|v\right|^{2}+\alpha\right)dv<\infty,
\end{equation}
for any $\alpha \in \mathbb R$. Furthermore, the condition
\[
\int_{\mathbb{R}^{D}}\int_{\mathbb{R}^{D}}\beta dxdv:=M\equiv1,
\]
where $M$ is the (normalized) mass of the system, can be used to compute
$\eta$.

We have
\[
\beta\left(x,v\right)=F\left(\frac{1}{2}\left|v\right|^{2}+V\left(x\right)-\eta-\frac{1}{2}\frac{\triangle\sqrt{\varrho\left(x\right)}}{\sqrt{\varrho\left(x\right)}}\right),
\]
and therefore we obtain the following integral equation for $\varrho$:
\[
\varrho\left(x\right)=\int_{\mathbb{R}^{D}}F\left(\frac{1}{2}\left|v\right|^{2}+V\left(x\right)-\eta-\frac{1}{2}\frac{\triangle\sqrt{\varrho\left(x\right)}}{\sqrt{\varrho\left(x\right)}}\right)dv.
\]
Hence,
\[
\varrho\left(x\right)=A\left(V\left(x\right)-\eta-\frac{1}{2}\frac{\triangle\sqrt{\varrho\left(x\right)}}{\sqrt{\varrho\left(x\right)}}\right),
\]
from which we obtain the equation
\begin{equation}
-\frac{1}{2}\frac{\triangle\sqrt{\varrho\left(x\right)}}{\sqrt{\varrho\left(x\right)}}+V\left(x\right)-A^{-1}\left(\varrho\left(x\right)\right)=\eta.\label{eq:S1}
\end{equation}
along with
\begin{equation}
\int_{\mathbb{R}^{D}}\varrho\left(x\right)dx=1.\label{eq:C_S1}
\end{equation}

To proceed further, we now restrict our attention to probability measures. For the rest of this section, and for simplicity of notation, for any probability measure, $\mu$, let us define $\mathcal{F}\left(\mu\right)$ as one eighth of the Fisher information, i.e.,
\[
\mu\in\mathcal{P}\left(\mathbb{R}^{D}\right)\rightarrow\mathcal{F}\left(\mu\right):=\left\{ \begin{array}{ll}
{1 \over 2} \int_{\mathbb{R}^{D}} \bigl|\nabla ( \sqrt \varrho) \bigr|^2dx, & \textrm{if }\mu=\varrho\mathcal{L}^{D}, \; \hbox{and} \; \sqrt \varrho \in W^{1,2}(\mathbb R^D)\\
  & \\ 
\infty, & \textrm{otherwise}.
\end{array}\right.
\] 
The properties of $\mathcal F$ can also be studied through the convex lower semicontinuous function $L:\mathbb{R}\times\mathbb{R}^{D}\rightarrow\left[0,+\infty\right]$
defined by
\begin{equation}
L\left(\varrho,\xi\right):=\left\{ \begin{array}{cl}
\frac{\left|\xi\right|^{2}}{2\varrho}, & \textrm{if }\varrho>0\\
0, & \textrm{if }\xi=\vec{0}\textrm{ and }\varrho=0\\
\infty, & \textrm{if }(\xi\neq\vec{0}\textrm{ and }\varrho=0)\textrm{ or }\left(\varrho<0\right)\textrm{ or }\left(\varrho=\infty\right).
\end{array}\right.\label{eq:def_L}
\end{equation}
If $\mu\in\mathcal{P}\left(\mathbb{R}^{D}\right)$ then 
\begin{equation}\label{eq:fisher3}
\mathcal{F}\left(\mu\right):=\left\{ \begin{array}{ll}
\frac{1}{4}\int_{\mathbb{R}^{D}}L\left(\varrho,\nabla\varrho\right)dx, & \textrm{if }\mu=\varrho\mathcal{L}^{D},  \; \hbox{and} \; L\left(\varrho,\nabla\varrho\right) \in L^{1}(\mathbb R^D)\\
& \\
\infty & \textrm{otherwise},
\end{array}\right.
\end{equation}

\begin{remark}\label{re:properties-of-A} Since $F$ is monotone, its set of discontinuity is countable and will be denoted by $\{t_n\}_{n=1}^\infty$
\begin{enumerate} 
\item[(i)] The infimum of $F$ must be $0$, otherwise we would have $A \equiv \infty.$   
\item[(ii)] We exploit (i) and the dominated convergence theorem to obtain 
\[ 
\lim_{\alpha \rightarrow \infty} A(\alpha)= \lim_{\alpha \rightarrow \infty} \int_{\mathbb{R}^{D}}F\Bigl(\frac{1}{2}\left|v\right|^{2}+\alpha\Bigr)dv= \int_{\mathbb{R}^{D}} \biggl(\lim_{\alpha \rightarrow \infty} F\Bigl(\frac{1}{2}\left|v\right|^{2}+\alpha\Bigr)\biggr) dv=0.
\] 
\item[(iii)]  Let $\bar \alpha \in \mathbb R$ and denote by $S_r(0)$ the sphere of radius $r$ centered at the origin. If $r_n^2+2\bar \alpha=2 t_n$, then the union of $N(\bar \alpha):=\cup_{n=1}^\infty S_{r_n}(0)$ is a set of null Lebesgue measure and  
\[
\lim_{\alpha \rightarrow \bar \alpha} F\Bigl({|v|^2 \over 2}+ \alpha \Bigr)= F\Bigl({|v|^2 \over 2}+ \bar \alpha \Bigr)
\] 
for all $v \not\in N(\bar \alpha).$ Thus, as above, by the dominated convergence theorem, $\lim_{\alpha \rightarrow \bar \alpha} A(\alpha)=A(\bar \alpha).$ In other words, $A$ is continuous on $\mathbb R.$ 
\item[(iv)] Let $\lambda_0>0$ be the supremum of $F$. We have 
\begin{equation}\label{eq:lower-bound-A}
\liminf_{\alpha \rightarrow -\infty} {A(\alpha) \over (-\alpha)^D} \geq {\Bigl|\mathbb S^{D-1} \Bigr| \over 2^D}.
\end{equation} 
Hence, 
\begin{equation}\label{eq:lower-bound-A-second}
\lim_{\alpha \rightarrow -\infty} A(\alpha) =\infty.
\end{equation}
Indeed,  if $\alpha<-2$
\[
A(\alpha)= \Bigl|\mathbb S^{D-1} \Bigr| \int_0^\infty r^{D-1} F\Bigl({r^2 \over 2} +\alpha \Bigr) dr \geq  
\Bigl|\mathbb S^{D-1} \Bigr| \int_{-{\alpha \over 2}}^{-\alpha} r^{D-1} F\Bigl({r^2 \over 2} +\alpha \Bigr) dr.
\] 
Since $F$ decreases, we conclude that  
\[
A(\alpha) \geq  
\Bigl|\mathbb S^{D-1} \Bigr| \Bigl(-{\alpha \over 2} \Bigr)^{D} F\Bigl({\alpha^2 \over 2} +\alpha \Bigr) \geq 
\Bigl|\mathbb S^{D-1} \Bigr| \Bigl(-{\alpha \over 2} \Bigr)^{D} F\Bigl(-\alpha \Bigr), 
\] 
which implies (\ref{eq:lower-bound-A}). Thus, (\ref{eq:lower-bound-A-second}) holds.

\item[(v)]  By (i - iv), $A: \mathbb R \rightarrow (0,\infty)$ is a homeomorphism, and
\[ \lim_{s \rightarrow \infty}-A^{-1}(s)=\infty, \quad \lim_{s \rightarrow 0}-A^{-1}(s)=-\infty.\]   
\item[(vi)] Let $B \in C^1(0,\infty)$ be such that 
\begin{equation}\label{eq:property-of-B}
B'\left(s\right)=-A^{-1}\left(s\right).
\end{equation}  
Observe that since $-A^{-1}$ strictly increases, $B$ is strictly convex.

\item[(vii)] Let $b(s)=B'(s)$. Using first (v) and then (iv) we obtain    
\begin{equation}\label{eq:property-of-b-and0}
\limsup_{s \rightarrow \infty}{b(s) \over s^{1\over D}}= \limsup_{\alpha \rightarrow -\infty}\biggl({-\alpha \over A(\alpha)} \biggr)^{1\over D} \leq {2 \over \Bigl|\mathbb S^{D-1} \Bigr|^{1\over D}}=: {\bar \lambda_1 \over 2}.
\end{equation}  
Therefore, we can choose $T_1>1$ such that 
\begin{equation}\label{eq:property-of-b-and0.2}
0< b(s) \leq \bar \lambda_1 s^{1\over D} 
\end{equation}  
for all $s \in [T_1, \infty)$. Since $b(s)$ increases as $s$ increases, setting $\bar \lambda_2:=b(T_1)>0$ we have 
\begin{equation}\label{eq:property-of-b-and0.3}
b(s) \leq \bar \lambda_1 s^{1\over D} +\bar \lambda_2 
\end{equation} 
for any $s \in (0,\infty).$
\item[(viii)] Suppose that $\lim_{s \rightarrow 0^+}B(s)$ exists. Since $B$ is defined up to additive constant, we can set $B(0)=0$ such that   
\begin{equation}\label{eq:property-of-b-and2}
B(t)=\int_0^t b(s) ds.
\end{equation} 
By (\ref{eq:property-of-b-and0.3}),  
\begin{equation}\label{eq:property-of-b-and2.1}
sb(s)  \leq \bar \lambda_1 s^{1 +{1 \over D}} + \bar \lambda_2 s
\end{equation} 
and 
\begin{equation}\label{eq:property-of-B3}
B(s) \leq \lambda_1  \Bigl(  s^{1 +{1 \over D} }+  s \Bigr),
\end{equation}  
for any $s \in (0,\infty).$ We have set 
\[
\lambda_1:= \max\{\bar \lambda_1, \bar \lambda_2  \}.
\]

\end{enumerate} 
\end{remark}
%
%

\begin{lemma}\label{le:properties-of-b} Suppose that $b$ and  $B$ are as in Remark \ref{re:properties-of-A} and $B(0)=0.$  Then

\begin{enumerate} 
\item[(i)]   the infimum of $B(s)$ is finite.
\item[(ii)] The infimum of $sb(s)$  is finite and for any $s>0$ we have 
\[
sb_-(s) \leq B_-(s).
\]
\end{enumerate}  
\end{lemma} 

\begin{proof}  If $B^*$ denotes the Legendre transform of $B$, then, by the fact that $B(0)=0$, we have 
\begin{equation}\label{eq:property-of-B4}
B^* \geq 0.
\end{equation} 

(i) Since by Remark \ref{re:properties-of-A} $\lim_{s \rightarrow \infty} b(s)=\infty$, there exists $s_0$ such that $b >0$ on $[s_0,\infty)$. Thus, $B$ is bounded below on  $[s_0,\infty)$ by $B(s_0).$ Since $B$ is continuous on $[0,s_0]$ we conclude that it is also bounded below there. Consequently, there exists $\lambda_b<0$ such that $B \geq -\lambda_b.$  

(ii) Let $s>0$ and set $\alpha=B'(s)=b(s).$ Since  
\[
sb_+(s)-sb_-(s)=sb(s)= B(s)+B^*(\alpha)=B_+(s)-B_-(s)+B^*(\alpha)
\] 
we conclude that 
\[
sb_-(s)+ B_+(s)+B^*(\alpha)= sb_+(s) +B_-(s).
\] 
Since by (\ref{eq:property-of-B4}) $B^* \geq 0$, we conclude the proof.
\end{proof}

\begin{example}Examples include  
\[
F(t)=e^{-t}, \quad A(t)= Ce^{-t}, \quad b(s)=\ln \bigl({s \over C}\bigr), \quad B(s)= s \ln \bigl({s \over C}\bigr) -s , 
\]
where
\[
C:=|\mathbb S^{D-1}| \int_0^\infty r^{D-1} e^{- r^2 \over 2} dr.
\]
\end{example}

In general, if $B$ satisfies (\ref{eq:property-of-B}), then, by Remark \ref{re:properties-of-A} (v), we have 
\begin{equation}\label{eq:properties-of-B}
\lim_{s \rightarrow \infty}{B(s) \over s}=\infty.
\end{equation}
We shall assume that 
\begin{equation}\label{eq:properties-of-B2}
B(0):=\lim_{s \rightarrow 0^+} B(s) \quad \hbox{exists}.
\end{equation} 

Define 
\[ 
\varrho_\infty:=A\left(V \right) 
\] 
and assume that 
\begin{equation}\label{eq:properties-of-rho-m}
\varrho_\infty \mathcal L^D \in \mathcal P_2(\mathbb R^D) \quad \hbox{and} \quad  f^\infty:=B(\varrho_\infty)+ \varrho_\infty V \in L^1(\mathbb R^D).  
\end{equation}

\begin{remark}\label{re:convexity} By the convexity of $B$, $B(s) \geq B(s_0)+b(s_0)(s-s_0)$  and if $s>0$, then $B'(s)s=B(s)+B^*\bigl(B'(s) \bigr).$ Hence, 
\begin{enumerate} 
\item[(i)] if $\varrho: \mathbb R^D \rightarrow [0,\infty]$ is a Borel function 
\[
B(\varrho)+V \varrho \geq B(\varrho_\infty)+V \varrho_\infty +\Bigl(b(\varrho_\infty)+V\Bigr)(\varrho-\varrho_\infty)=B(\varrho_\infty)+V \varrho_\infty=  f^\infty.
\]
Consequently, due to (\ref{eq:properties-of-rho-m}), 
\begin{equation}\label{eq:upper-bound}
\Bigl(B(\varrho)+V \varrho\Bigr)_-\leq f^\infty_-  \quad \hbox{and} \quad  B_-(\varrho) \leq V \varrho -f^\infty.
\end{equation}
Hence, the functional  
\[
P\left(\varrho\right):=\int_{\mathbb{R}^{D}}\left(V\varrho+B\left(\varrho\right)\right)dx.
\] 
is meaningful and achieves its minimum at $\varrho_\infty.$    
\item[(ii)] We use the first inequality in (i) to conclude that for $\varrho>0$ we have   
\[
B'(\varrho) \varrho + V \varrho = B(\varrho)+ B^* \bigl(B'(\varrho) \bigr)+ V \varrho  \geq f^\infty + B^* \bigl(B'(\varrho) \bigr).
\] 
\item[(iii)] In particular, a consequence of (ii) is that, since $B(0)=0$ implies $B^* \geq 0$, Lemma \ref{le:properties-of-b} and (\ref{eq:upper-bound}) imply 
\[
\varrho b_-(\varrho) \leq V \varrho - f^\infty. 
\]
\end{enumerate}  
\end{remark}

\begin{lemma}\label{lem:upperbound} Let $\varrho: \mathbb R^D \rightarrow [0,\infty]$ be a Borel function. Then 
\begin{enumerate} 
\item[(i)] 
\[
|B(\varrho)| \leq \lambda_1 \Bigl(\varrho^{1+{1 \over D}}+ \varrho \Bigr)+V \varrho-f^\infty.
\]
\item[(ii)] At the point where $\varrho>0$, we have 
\[
\varrho |b(\varrho)| \leq \bar \lambda_1 \varrho^{1+{1 \over D}}+ \bar \lambda \varrho+V \varrho-f^\infty
\]
\end{enumerate}  
\end{lemma} 

\begin{proof} We combine (\ref{eq:property-of-B3}) and (\ref{eq:upper-bound}) to obtain (i).  The proof of (ii) follows by combining (\ref{eq:property-of-b-and0.3}) and Remark \ref{re:convexity}. 
\end{proof}

Now define the functional $E: \mathcal P(\mathbb R^D) \rightarrow (-\infty, \infty]$ by 
\begin{equation}
E\left(\mu\right):=\left\{ \begin{array}{ll}
\mathcal{F}\left(\mu\right)+P(\varrho), & \textrm{if }\mu=\varrho\mathcal{L}^{D},\\
\infty, & \textrm{otherwise.}
\end{array}\right.\label{eq:func_E}
\end{equation}

\begin{lemma}\label{lem:conv} Assume (\ref{eq:properties-of-rho-m}) holds. On its proper domain, the functional $E$ defined in (\ref{eq:func_E}) is strictly convex and bounded below. Furthermore, $E$ is lower semicontinuous for the narrow convergence on $\mathcal P(\mathbb R^D).$
\end{lemma}
\begin{proof} As $\mathcal F \geq 0$, we use Remark \ref{re:convexity} to conclude that $E\left(\mu\right)\geq P(\varrho_\infty)$. Furthermore, we use (\ref{eq:upper-bound}) to conclude that the proper domain of $E$ is the intersection of the proper domains of $\mathcal F$ and $P.$ The strict convexity of $B$ implies that of $P$ on its proper domain. 

To show that $E$ is lower semicontinuous for the narrow convergence on $\mathcal P(\mathbb R^D)$ it suffices to show that $\mathcal F$ and $P$ are both lower semicontinuous. Let $(\mu_n)_n \subset \mathcal P(\mathbb R^D)$ be a sequence that converges to $\mu$ narrowly and assume that  
\[
\sup_n E(\mu_n)<\infty.
\] 
By Lemma 2.2 of \cite{matthes2009family}, there exist $\varrho_n: \mathbb R^D \rightarrow [0,\infty]$ and $\varrho: \mathbb R^D \rightarrow [0,\infty]$ such that 
\[
\mu_n=\varrho_n \mathcal L^D, \quad \mu=\varrho \mathcal L^D, \quad \varrho_n, \varrho \in W^{1,1}_{loc}(\mathbb R^D), \quad {|\nabla \varrho_n| \over \sqrt \varrho_n}, {|\nabla \varrho| \over \sqrt \varrho} \in L^2(\mathbb R^D),
\] 
\begin{equation}\label{eq:upper-bound2}
\liminf_n \mathcal F(\mu_n) \geq \mathcal F(\mu),
\end{equation}
$(\sqrt \varrho_n)_n$ converges to $\sqrt \varrho$, strongly in $L^2(\mathbb R^D)$ and weakly in $W^{1,2}(\mathbb R^D).$ Thus, every subsequence of $(\varrho_n)_n$ admits itself a subsequence which converges  almost everywhere to $\varrho$. By (\ref{eq:upper-bound}), $B(\varrho_n)+V\varrho_n+f^\infty_- \geq 0$. Therefore, we can apply Fatou's Lemma to obtain
\[
\liminf_{n \rightarrow \infty} \int_{\mathbb R^D} \Bigl(B(\varrho_n)+V\varrho_n+f^\infty_- \Bigr) dx \geq 
\int_{\mathbb R^D} \Bigl(B(\varrho)+V\varrho+f^\infty_- \Bigr) dx.
\]
Then, 
\begin{equation}\label{eq:upper-bound3}
\liminf_n P(\varrho_n) \geq P(\varrho). 
\end{equation}
By (\ref{eq:upper-bound2}) and (\ref{eq:upper-bound3}), $E$ is lower semicontinuous.

Convexity of $\mathcal F$ follows from that of $L.$ Consequently, $E$ is strictly convex on its proper domain. 

\end{proof}

We shall see that solutions of (\ref{eq:S1}) can be obtained by minimizing $E$.

\begin{remark}\label{re:narrow-compact}
Recall that a set $\mathcal{K}\subset\mathcal{P}\left(\mathbb{R}^{D}\right)$ is tight if
\begin{equation}
\forall\varepsilon>0\quad\exists K_{\varepsilon}\textrm{ compact in }\mathbb{R}^{D}\textrm{ such that}\quad\mu\left(\mathbb{R}^{D}\backslash K_{\varepsilon}\right)\leq\varepsilon\quad\forall\mu\in\mathcal{K}.\label{eq:tig_cond}
\end{equation}
Moreover, it can be verified that (\ref{eq:tig_cond}) is equivalent
to the following integral condition (cf. Remark 5.1.5 in \cite{ambrosio2008gradient}):
there exists a function $\vartheta:\mathbb{R}^{D}\rightarrow\left[0,+\infty\right]$,
whose sublevels $\left\{ x\in\mathbb{R}^{D}\left|\vartheta\left(x\right)\leq c\right.\right\} $
are compact in $\mathbb{R}^{D}$, such that
\[
\underset{\mu\in\mathcal{K}}{\sup}\int_{\mathbb{R}^{D}}\vartheta\left(x\right)d\mu\left(x\right)<+\infty.
\]
\end{remark}

\begin{lemma}\label{lem:phi_V} Consider a strictly convex function $B:\mathbb R \rightarrow\left[0,+\infty\right]$, with $B\left(\infty\right)=\infty$ and differentiable on $(0,\infty)$. Suppose there are strictly positive Borel functions $\varrho_{\infty,\alpha}$ and $\varrho_{\infty}$ such that, on the set where these expressions are positive, we have 
\begin{equation}
-B'\left(\varrho_{\infty} \right)=V \label{eq:r1}
\end{equation}
and for some $0<\alpha<1$ 
\begin{equation}
-B'\left(\varrho_{\infty,\alpha} \right)=\alpha V,\label{eq:r2}
\end{equation}
and 
\begin{equation}
B\left(\varrho_{\infty,\alpha}\left(x\right)\right)+\alpha V\left(x\right)\varrho_{\infty,\alpha}\left(x\right)\in L^{1}\left(\mathbb{R}^{D}\right).\label{eq:C_K1} 
\end{equation} 
Assume $V: \mathbb R^D  \rightarrow \mathbb R$ is a Borel function which satisfies  
\begin{equation}\label{eq:growth-on-V}
\underset{\left|x\right|\rightarrow\infty}{\lim}V\left(x\right)=+\infty
\end{equation}
and there exists $\underline{V}\in\mathbb{R}$ such that $V\left(x\right)\geq\underline{V}$ for almost every $x\in\mathbb{R}^{D}.$ For any $K>0$ there exists a constant $\tilde K>0$ such that if $\varrho\in L^{1}\left(\mathbb{R}^{D}\right)$ is nonnegative and 
\begin{equation}\label{eq:as_1}
\int_{\mathbb{R}^{D}}\left(B\left(\varrho\left(x\right)\right)+V\left(x\right)\varrho\left(x\right)\right)dx\leq K,
\end{equation}
then,
\[
\int_{\mathbb{R}^{D}}V\left(x\right)\varrho\left(x\right)dx\leq\tilde{K}.
\]
\end{lemma}
\begin{proof} If (\ref{eq:as_1}) holds, then 
\begin{align}
K\geq &  \int_{ \mathbb R^D }\left(B\left(\varrho\left(x\right)\right)+\alpha V\left(x\right)\varrho\left(x\right)+\left(1-\alpha\right)V\left(x\right)\varrho\left(x\right)\right)dx\nonumber \\
= & \int_{ \mathbb R^D }\left(B\left(\varrho\left(x\right)\right)-B'\left(\varrho_{\infty,\alpha}\left(x\right)\right)\varrho\left(x\right)+\left(1-\alpha\right)V\left(x\right)\varrho\left(x\right)\right)dx,\label{eq:inq_K_1}
\end{align}
where we used (\ref{eq:r2}) for the last expression. Since 
\[
B\left(\varrho \right)\geq B\left(\varrho_{\infty,\alpha} \right)+B'\left(\varrho_{\infty,\alpha} \right)\left(\varrho -\varrho_{\infty,\alpha} \right), 
\]
(\ref{eq:inq_K_1}) implies 
\begin{align}
K\geq & \int_{\mathbb{R}^{D}}\left(B\left(\varrho_{\infty,\alpha}\left(x\right)\right)-B'\left(\varrho_{\infty,\alpha}\left(x\right)\right)\varrho_{\infty,\alpha}\left(x\right)+\left(1-\alpha\right)V\left(x\right)\varrho\left(x\right)\right)dx\nonumber \\
= & \int_{\mathbb{R}^{D}}\left(B\left(\varrho_{\infty,\alpha}\left(x\right)\right)+\alpha V\left(x\right)\varrho_{\infty,\alpha}\left(x\right)+\left(1-\alpha\right)V\left(x\right)\varrho\left(x\right)\right)dx\nonumber \\
= & C+\left(1-\alpha\right)\int_{\mathbb{R}^{D}}V\left(x\right)\varrho\left(x\right)dx,\label{eq:in_K_2}
\end{align}
where, due to (\ref{eq:C_K1}), we have set 
\[
C:=\int_{\mathbb{R}^{D}}\left(B\left(\varrho_{\infty,\alpha}\left(x\right)\right)+\alpha V\left(x\right)\varrho_{\infty,\alpha}\left(x\right)\right)dx.
\]
By (\ref{eq:in_K_2}), 
\[
\int_{\mathbb{R}^{D}}V\left(x\right)\varrho\left(x\right)dx\leq\frac{K-C}{1-\alpha}=:\tilde{K}.
\]
\end{proof}

\begin{remark}
Let $F\left(s\right)=e^{-s}$, which implies $B\left(s\right)=s\ln s$. Then, all the assumptions
in Lemma \ref{lem:phi_V} are satisfied if we have $e^{-\alpha V\left(x\right)}\in L^{1}\left(\mathbb{R}^{D}\right)$
for some $0<\alpha<1$.
\end{remark}

\begin{theorem} \label{thm:var_eqn} Assume $V: \mathbb R^D \rightarrow \mathbb R$ is a Borel function, bounded below and satisfying (\ref{eq:growth-on-V}). Suppose $F:\mathbb{R}\rightarrow\mathbb{R}^{+}$ is strictly decreasing and is such that for any $\alpha \in \mathbb R$ the function in (\ref{eq:defn-of-A}) assumes only finite values. Suppose further that $B \in C^1(0,\infty) \cap C\bigl([0,\infty) \bigr)$ is such that $B'=-A^{-1}$, $B(0)=0$ and (\ref{eq:properties-of-B2}) holds. Finally, assume that $\lim_{s \rightarrow 0} sB'(s)=0,$   $B(\varrho_\infty)+V \varrho_\infty \in L^1(\mathbb R^D)$,  and (\ref{eq:C_K1}) holds. If $E\not\equiv\infty$, then the minimization problem
\[
\underset{\mu\in\mathcal{P}_{2}\left(\mathbb{R}^{D}\right)}{\mathrm{argmin}}E\left(\mu\right),
\]
has a unique solution, $\mu_s=\varrho_{s}\mathcal{L}^{D}$. Setting 
\[
 \eta_s:= {1 \over 2} ||\nabla  \sqrt {\varrho_s}||_{L^2}^2+ \int_{\mathbb R^D} \bigl(B'(\varrho_s)+V \bigr) \varrho_s dx, 
 \] 
we have in the weak sense  
\begin{equation}\label{eq:EulerLagrange}
-{1 \over 2} \triangle \varrho_s + |\nabla \sqrt \varrho_s|^2   + 2 \bigl(B'(\varrho_s)+V \bigr) \varrho_s= \eta_s \varrho_s,
\end{equation}
which can be interpreted as (\ref{eq:S1}) 
\end{theorem}

\begin{proof} {\it Part I: Existence and uniqueness of a minimizer.}  Let $\left\{ \mu_{n}\right\} _{n\in\mathbb{N}}$ be a minimizing sequence
of $E\left(\mu\right)$, i.e.,
\[
\underset{n\rightarrow\infty}{\lim}E\left(\mu_{n}\right)=\underset{\mu\in\mathcal{P}_{2}\left(\mathbb{R}^{D}\right)}{\inf}E\left(\mu\right).
\]
Since both $P$ and $\mathcal F$ are bounded below,  
\[
\sup_n \mathcal F(\mu_n)<\infty. 
\]
By Lemma 2.2 of \cite{matthes2009family},  there exist $\varrho_n: \mathbb R^D \rightarrow [0,\infty]$  such that $\mu_n=\varrho_n \mathcal L^D.$ We have 
\[
\sup_n P(\varrho_n)<\infty 
\]
and hence, Lemma \ref{lem:phi_V} implies 
\[
\sup_n \int_{\mathbb R^D} V \varrho_n dx <\infty. 
\]
Thus, by Remark \ref{re:narrow-compact}, $\{\mu_n\}_n$ is pre--compact for the narrow convergence. Extracting a subsequence if necessary, we assume without loss of generality that $\{\mu_n\}_n$ converges narrowly to some $\mu_s \in \mathcal P_2(\mathbb R^D).$ By Lemma 2.2 of \cite{matthes2009family}, there exists $\varrho_s: \mathbb R^D \rightarrow [0,\infty]$ such that 
\[
\mu_n=\varrho_n \mathcal L^D, \quad \mu_s=\varrho_s \mathcal L^D, \quad \varrho_n, \varrho_s \in W^{1,1}_{loc}(\mathbb R^D), \quad {|\nabla \varrho_n| \over \sqrt \varrho_n}, {|\nabla \varrho_s| \over \sqrt \varrho_s} \in L^2(\mathbb R^D).
\] 
Furthermore, $(\sqrt \varrho_n)_n$ converges to $\sqrt \varrho_s$, strongly in $L^2(\mathbb R^D)$ and weakly in $W^{1,2}(\mathbb R^D).$ By Lemma \ref{lem:conv}, $E$ is lower semicontinuous for the narrow convergence and hence, $\mu_s$ minimizes $E$ over $\mathcal P_2(\mathbb R^D)$. 

Uniqueness of $\mu_s$ follows from the strict convexity property of $E$ on its domain (cf. Lemma \ref{lem:conv}). \\ 

 {\it Part II: Properties of the minimizer.} 
Since 
\begin{equation}\label{eq:expansion-0.1}
P(\varrho_s) \leq \underset{\mu\in\mathcal{P}_{2}\left(\mathbb{R}^{D}\right)}{\inf}E\left(\mu\right), 
\end{equation} 
we use the last statement in Lemma \ref{lem:phi_V} and the fact that $V$ is bounded below to deduce that  
\begin{equation}\label{eq:expansion0}
\varrho_s |V| \in L^1(\mathbb R^D). 
\end{equation}
By Remark \ref{re:convexity} 
\begin{equation}\label{eq:expansion0.1}
B(\varrho_s)_- \leq  \varrho_s V -B(\varrho_\infty) -\varrho_\infty V \in L^1(\mathbb R^D). 
\end{equation}
Thus, combining (\ref{eq:expansion-0.1}), (\ref{eq:expansion0}) and (\ref{eq:expansion0.1}) we conclude that 
\begin{equation}\label{eq:expansion0.2}
B(\varrho_s) \in L^1(\mathbb R^D). 
\end{equation}

 {\it Part III: The Euler--Lagrange equations.} Let $v \in C_c^\infty(\mathbb R)$ and set 
\[
u_0=\sqrt \varrho_s, \quad u_\epsilon:={u_0+ \epsilon u_0 v \over ||u_0+ \epsilon u_0 v||_{L^2}}, \quad \mu^\epsilon=u_\epsilon^2 \mathcal L^D. 
\] 
We have 
\begin{equation}\label{eq:expansion1}
u_\epsilon^2= u_0^2 +2\epsilon u_0^2 a(v) +\epsilon^2 u_0^2a_\epsilon(v),
\end{equation}
where
\[a(v):= v- \int_{\mathbb R^D} u_0^2 vdx  \quad \hbox{and} \quad  \sup_{0 < |\epsilon|<1}||b_\epsilon(v) ||_\infty<\infty.
\] 
We set 
\[
S:= \Bigl(1+  ||2 a(v)||_{L^\infty} + \sup_{|\epsilon| \leq 1} || a_\epsilon(v))||_{L^\infty}\Bigr)^{1\over 2}.
\] 

We have 
\[ 
\int_{\mathbb R^D} V u_\epsilon^2 dx -\int_{\mathbb R^D} V u_0^2 dx = 
\epsilon \int_{ \mathbb R^D} V u_0^2 \Bigl( 2a(v) +\epsilon a_\epsilon(v) \Bigr) dx.
\]
Therefore, exploiting (\ref{eq:expansion0}) we can apply the dominated convergence theorem to obtain 
\begin{equation}\label{eq:expansion1.5} 
{d \over d \epsilon} \int_{\mathbb R^D} V u_\epsilon^2 dx \biggl|_{\epsilon=0} =  \int_{ \mathbb R^D} 2a(v) V u_0^2 dx
\end{equation}

If $D\geq 3$, then 
\[
{2D \over D-2} \geq {2+ {2 \over D}}. 
\] 
Since by the Sobolev Embedding theorem $W^{1,2}(\mathbb R^D)  \subset L^{{2D \over D-2}}(\mathbb R^D)$, we conclude that $ W^{1,2}(\mathbb R^D)   \subset L^{{2+ {2 \over D}}}(\mathbb R^D).$ The latter inclusion remains true when $D \in \{1, 2\}.$ Consequently, $u_0 \in L^{{2+ {2 \over D}}}(\mathbb R^D)$ and then, $F^\infty \in L^1(\mathbb R^D)$ if we set 
\[
F^\infty:=  \lambda_1 \Bigl(  u_0^{2 +{2 \over D}} S^{2 +{2 \over D}}+  u_0^{2}  S^2\Bigr) + |V|u_0^2S^2-f^\infty.
\]

By Lemma \ref{lem:upperbound} 
\begin{equation}\label{eq::ante-expand2} 
|B(u_\epsilon^2)| \leq \lambda_1 \Bigl(  u_\epsilon^{2 +{2 \over D}} +  u_\epsilon^{2}  \Bigr) + Vu_\epsilon^2-f^\infty 
\leq \lambda_1 \Bigl(  (u_0 S)^{2 +{2 \over D}} +  (u_0 S)^{2}  \Bigr) + |V|(u_0 S)^2-f^\infty  = F^\infty.
\end{equation}

Let $\theta_\epsilon: \mathbb R^D \rightarrow (0,1)$ be such that if $u_0>0$ we have the first order expansion
\[
B(u_\epsilon^2)- B(u_0^2)= (u_\epsilon^2-u_0^2) B'\Bigl(u_0^2 + \theta_\epsilon ( (u_\epsilon^2-u_0^2)) \Bigr).
\] 
This means that 
\[
B(u_\epsilon^2)- B(u_0^2)= \epsilon u_0^2 \Bigl(2 a(v) +\epsilon a_\epsilon(v) \Bigr) 
B'\Bigl(\bigl(u_0\Theta_\epsilon\bigr)^2 \Bigr), 
\] 
where   
\[
\Theta_\epsilon:= \Bigl(1 + \epsilon  \theta_\epsilon \bigl[2 a(v) +\epsilon a_\epsilon(v) \bigr]\Bigr)^{1\over 2}.
\]
Reorganizing the expession, we have   
\begin{equation}\label{eq::ante-expand3}
{B(u_\epsilon^2)- B(u_0^2) \over \epsilon}= 
{ \Bigl(2 a(v) +\epsilon a_\epsilon(v) \Bigr) \over 1 + \epsilon  \theta_\epsilon \bigl[2 a(v) +\epsilon a_\epsilon(v) \bigr]}
\bigl(u_0 \Theta_\epsilon\bigr)^2 
B'\Bigl(\bigl(u_0\Theta_\epsilon\bigr)^2 \Bigr).
\end{equation} 
This, together with Lemma \ref{lem:upperbound}, imply 
\[
\biggl|{B(u_\epsilon^2)- B(u_0^2) \over \epsilon}\biggr| \leq 
\biggl( \bar \lambda_1 \bigl(u_0\Theta_\epsilon\bigr)^{2 +{2 \over D}} + \bar \lambda_2 \bigl(u_0\Theta_\epsilon\bigr)^{2} +|V| \bigl(u_0\Theta_\epsilon\bigr)^{2}-f^\infty\biggr) 
{ \Bigl|2 a(v) +\epsilon a_\epsilon(v) \Bigr| \over 1 + \epsilon  \theta_\epsilon \bigl[2 a(v) +\epsilon a_\epsilon(v) \bigr]}.
\]
Thus, if $|\epsilon|$ is small enough so that $2\bigl|\epsilon  \theta_\epsilon \bigl[2 a(v) +\epsilon a_\epsilon(v) \bigr]\bigr| \leq 1$, then 
\begin{equation}\label{eq::ante-expand4}
\biggl|{B(u_\epsilon^2)- B(u_0^2) \over \epsilon}\biggr| \leq 
2 S^2\biggl( \bar \lambda_1 \bigl(u_0S\bigr)^{2 +{2 \over D}} + \bar \lambda_2 \bigl(u_0S\bigr)^{2} +V \bigl(u_0S\bigr)^{2}-f^\infty\biggr) \in L^1(\mathbb R^D).
\end{equation} 
Since $B(0)=0$ and $u_\epsilon\equiv 0$ on $\{u_0=0\},$ we conclude that 
\[
\int_{\mathbb R^D} {B(u_\epsilon^2)- B(u_0^2) \over \epsilon} dx= \int_{\{u_0>0\}} {B(u_\epsilon^2)- B(u_0^2) \over \epsilon} dx.
\] 
Due to (\ref{eq::ante-expand4}), we can apply the dominated convergence theorem to conclude that 
\[
{d \over d \epsilon}  \int_{\mathbb R^D} B(u_\epsilon^2) dx\bigg|_{\epsilon=0}= 
\int_{\{u_0>0\}} \lim_{\epsilon \rightarrow 0} {B(u_\epsilon^2)- B(u_0^2) \over \epsilon} dx.
\] 
We then let $\epsilon$ go to $0$ in  (\ref{eq::ante-expand3}) to deduce that 
\[
{d \over d \epsilon}  \int_{\mathbb R^D} B(u_\epsilon^2) dx\bigg|_{\epsilon=0}=  
2\int_{\{u_0>0\}} a(v)u_0^2  B'\bigl(u_0^2\bigr)dx.
\]
Taking into account the fact that $\lim_{s \rightarrow 0} s B'(s)=0$, we get  
\begin{equation}\label{eq::ante-expand5}
{d \over d \epsilon}  \int_{\mathbb R^D} B(u_\epsilon^2) dx\bigg|_{\epsilon=0}=  
2\int_{\mathbb R^D} a(v)u_0^2  B'\bigl(u_0^2\bigr)dx.
\end{equation}

Note that   
\begin{equation}\label{eq:expansion5} 
|\nabla u_\epsilon|^2=|\nabla u_0|^2+ 2\epsilon e(v) 
+ \epsilon^2 \Bigl( u_0^2+ |u_0 \nabla u_0| + |\nabla u_0|^2\Bigr)e_\epsilon(v), 
\end{equation}  
where 
\[
e(v):=  \langle \nabla u_0; \nabla (vu_0) \rangle  - |\nabla u_0|^2 \int_{\mathbb R^D} u_0^2 vdx  \quad \hbox{and} \quad  
 \sup_{0 < |\epsilon|<1}||e_\epsilon(v) ||_\infty<\infty
\] 
Hence, applying the dominated convergence theorem, we have
\begin{equation}\label{eq:expansion6} 
{d \over d \epsilon}  \int_{\mathbb R^D} |\nabla u_\epsilon|^2 dx\bigg|_{\epsilon=0} =\int_{\mathbb R^D} 2e(v) dx. 
\end{equation}  
We combine (\ref{eq:expansion1.5}), (\ref{eq::ante-expand5}) and (\ref{eq:expansion6}) to conclude that 
\[
{d \over d \epsilon} E(\mu^\epsilon) \bigg|_{\epsilon=0} = 
\int_{ \mathbb R^D } e(v) dx + 2 \int_{ \mathbb R^D } \bigl(B'(u_0^2)+V \bigr) u_0^2 a(v)  dx.
\]
Using the fact that $E(\mu^\epsilon)$ achieves its minimum at $\epsilon=0$, we conclude that  
\begin{equation}\label{eq:expansion7} 
0= {d \over d \epsilon} E(\mu^\epsilon) \bigg|_{\epsilon=0} = 
\int_{ \mathbb R^D } \Bigl( e(v)+ 2 \bigl(B'(u_0^2)+V \bigr) u_0^2 a(v)  \Bigr)dx.
\end{equation}
In other words, 
\[
\int_{ \mathbb R^D } \biggl( 
\langle \nabla u_0; \nabla (u_0 v) \rangle -l_0^2 u_0^2  v + 2 \bigl(B'(u_0^2)+V \bigr) u_0^2 v -2 u_0^2 l_1 v 
\biggr)dx=0, 
\] 
where 
\[
l_0:=||\nabla u_0||_{L^2}, \quad l_1:= \int_{\mathbb R^D} \bigl(B'(u_0^2)+V \bigr) u_0^2 dx. 
\] 
This implies that for all $v \in C_c^\infty(\mathbb R^D)$ 
\begin{align}
0= & \int_{ \mathbb R^D } \biggl( \langle u_0 \nabla u_0; \nabla  v \rangle +  \Bigl(|\nabla u_0|^2-l_0^2 u_0^2   + 2 \bigl(B'(u_0^2)+V \bigr) u_0^2  -2 u_0^2 l_1 \Bigr) v \biggr)dx 
\nonumber \\
= & \int_{ \mathbb R^D } \biggl( {1 \over 2}\langle \nabla \varrho_s ; \nabla  v \rangle +  \Bigl( |\nabla \sqrt \varrho_s|^2   + 2 \bigl(B'(\varrho_s)+V -l_1 -{l_0^2 \over 2} \bigr) \varrho_s  \Bigr) v \biggr)dx.  \label{eq:eulerLeq}
\end{align}
This means that (\ref{eq:EulerLagrange}) holds in the distributional sense.
\end{proof}

\begin{definition} Given $G: \mathcal P(\mathbb R^D) \rightarrow (-\infty, \infty]$, we define $G^*$ on the set of Borel functions $W: \mathbb R^D \rightarrow (\infty, \infty]$ which is bounded below, by  
\[
G^{*}\left(W\right)=\underset{\mu}{\sup}\left\{ \int_{\mathbb{R}^{D}} W(x) \mu (dx)-G\left(\mu\right)\; | \; \mu \in \mathcal P(\mathbb R^D)\right\}.
\] 
We refer to $G^*$ as the Legendre transform of $G.$
\end{definition}

The next result follows immediately from the definition of the Legendre tranform. 
\begin{lemma}\label{lem:Legendre2} If $V$ and $E$ are as in Theorem \ref{thm:var_eqn} and for any $\mu \in \mathcal P(\mathbb R^D)$ we define
\[
G\left(\mu\right):=\left\{ \begin{array}{ll}
\mathcal{F}\left(\mu\right)+\int_{\mathbb{R}^{D}}B\left(\varrho\right)dx, & \textrm{if }\mu=\varrho\mathcal{L}^{D},\\
& \\
\infty, & \textrm{otherwise,}
\end{array}\right.
\]
then  
\[
-G^{*}\left(-V\right)=\underset{\varrho}{\inf}E\left(\varrho \mathcal L^D\right).
\]
\end{lemma}

\begin{remark}  The conclusions in Theorem \ref{thm:var_eqn} remain valid if we  replace $\mathbb R^D$ by the torus $\mathbb T^D$. We keep the same assumptions on $B$ and $b$, but on $V$  we only assume that $V: \mathbb R^D \rightarrow \mathbb R$ is a Borel function bounded below, skipping (\ref{eq:growth-on-V}).  
\end{remark}

%
%
%
\section{Moreau-Yosida approximation}\label{sec:moreau-app} 
In the remainder of this paper, we develop an approximative version of the kinetic Bohmian equation with the aim of applying the results obtained in \cite{ambrosio2008hamiltonian}. As we shall see, this approximative version allows us to get around one of the main difficulties of the kinetic Bohmian equation when studied in the context of Wasserstein Hamiltonian flows: the lack of $\lambda-$convexity of the corresponding Hamiltonian.

We assume throughout this section that $d$ is an integer with $d \geq 1$ and  $D \in \{d, 2d\}$. We also assume that $\Phi: \mathcal P_2(\mathbb R^D) \rightarrow [0, \infty]$ is proper and lower semicontinuous with respect to the narrow convergence on bounded subsets of $ \mathcal P_2(\mathbb R^{D}).$  If $D=2d$ we assume that 
\[
\emptyset \not = D(\Phi) \subset \Bigl\{\mu \in \mathcal P_2(\mathbb R^{2d})\; | \; \pi^1_\# \mu \in \mathcal P_2^r(\mathbb R^d)  \Bigr\}.
\]
Finally, when $D=d$, we assume that 
\[
\emptyset \not = D(\Phi) \subset  \mathcal P_2^r(\mathbb R^d).
\]

For $\tau>0$ and $\mu \in \mathcal P_2(\mathbb R^{2d})$, we define the Moreau--Yosida approximation of $\Phi$ by
\begin{equation}\label{eq:dec20.2015.1}
\Phi_\tau(\mu)= \inf_{\nu} \Bigl\{{1 \over 2\tau}W_2^2(\mu, \nu)+ \Phi(\nu) \Bigr\}. 
\end{equation} 
We shall use the function   
\[
M_2(\mu)= {1 \over 2} \int_{\mathbb R^D} |z|^2 \mu(dz).
\] 
We fix  $\nu_\ast \in D(\Phi)$  and set 
\[
C_\tau:={2 \over \tau} M_2(\nu_\ast) +\Phi(\nu_\ast).
\] 
\begin{remark}\label{re:moreau-fisher}  Existence of a solution in (\ref{eq:dec20.2015.1}) is a standard result due to the fact that $\Phi$ is lower semicontinuous for the narrow convergence. Moreover, we define the set of minimizers 
\[
J_\tau^\Phi(\mu):= \Bigl\{\nu \in \mathcal P_2(\mathbb R^D) \; | \;  \Phi_\tau(\mu)={1 \over 2\tau}W_2^2(\mu, \nu)+ \Phi(\nu) \Bigr\}
\] 
By abuse of notation, we denote by $\mu^\tau$ any element of $J_\tau^\Phi(\mu).$ 

If $\mu \in \mathcal P_2^r(\mathbb R^D)$, then $W_2^2(\mu, \cdot)$ is strictly convex along geodesics of the $L^1$--metric and hence, since in addition $\Phi$ is convex, $J_\tau^\Phi(\mu)$ reduces to a single element  (cf., e.g., \cite{gianazza2009wasserstein} and \cite{villani2003topics}).  
\end{remark}

 \begin{lemma} \label{le:moreau-convex} The following hold: 
 \begin{enumerate}
\item[(i)] $-\Phi_\tau$ is $({-1 \over \tau})$--convex along geodesics of constant speed.
\item[(ii)] If $\mu \in \mathcal P_2(\mathbb R^D)$, then 
\[
0 \leq \Phi_\tau(\mu) \leq {1 \over  \tau}M_2(\mu) +C_\tau
\]  
\item[(iii)] Let $\mu_0, \mu \in \mathcal P_2(\mathbb R^D)$; let $G \in \Gamma_{o}(\mu_0, \mu^\tau_0)$ and denote by $G_{\mu_0}^{\mu^\tau_0}$ the  barycentric projection of $G$. Let $\bar G \in \Gamma_{o}(\mu_0, \mu).$ We have  
\[
\Phi_\tau(\mu) \leq \Phi_\tau(\mu_0)+ 
\int_{\mathbb R^D \times \mathbb R^D} \Bigl \langle  {w -G_{\mu_0}^{\mu^\tau_0}(w)\over \tau} ; z-w \Bigr\rangle \bar G(dw, dz)+ {1 \over 2 \tau}W_2^2(\mu, \mu_0).
\]
\item[(iv)] We conclude that  
\[
{ {\bf id}-G_{\mu_0}^{\mu^\tau_0}\over \tau} \in \bar{\partial} \Phi_\tau (\mu_0).
\]
\end{enumerate}
 \end{lemma}
\proof{} (i) Let $\mu_0, \mu_1 \in \mathcal P_2(\mathbb R^D)$ and let $(\mu_t)_t$ be a geodesic of constant speed connecting $\mu_0$ to $\mu_1.$  Fix $t \in (0,1)$ and let  $\mu_t^\tau \in \mathcal P_2(\mathbb R^D )$ be such that 
\begin{equation}\label{eq:dec28.2015.1}
\Phi_\tau(\mu_t)= {1 \over 2\tau}W_2^2(\mu_t, \mu^\tau_t)+ \Phi(\mu^\tau_t).
\end{equation} 
We have 
\[
\Phi_\tau(\mu_i) \leq {1 \over 2\tau}W_2^2(\mu_i, \mu^\tau_t)+ \Phi(\mu^\tau_t) \quad \forall \; i \in \{0,1\}.
\] 
Thus, 
\begin{equation}\label{eq:dec28.2015.2}
(1-t)\Phi_\tau(\mu_0)+t \Phi_\tau(\mu_1) \leq {1-t \over 2\tau}W_2^2(\mu_0, \mu^\tau_t)+ {t \over 2\tau}W_2^2(\mu_1, \mu^\tau_t)+ \Phi(\mu^\tau_t).
\end{equation} 
Since $-1/2W_2^2(\cdot, \mu^\tau_t)$ is $(-1)$--convex along geodesics of constant speed (cf., e.g., \cite{ambrosio2008gradient}), we conclude that 
\[
W_2^2(\mu_t, \mu^\tau_t)+t(1-t)W_2^2(\mu_0, \mu_1) \geq (1-t) W_2^2(\mu_0, \mu^\tau_t)+t W_2^2(\mu_1, \mu^\tau_t).
\]
This, along with (\ref{eq:dec28.2015.2}), yield 
\[
(1-t)\Phi_\tau(\mu_0)+t \Phi_\tau(\mu_1) \leq {1 \over 2\tau}W_2^2(\mu_t, \mu^\tau_t)+ {1 \over 2\tau}t(1-t)W_2^2(\mu_0, \mu_1)+\Phi(\mu_t). 
\] 
Therefore, by (\ref{eq:dec28.2015.1}) 
\[ %
(1-t)\Phi_\tau(\mu_0)+t \Phi_\tau(\mu_1) \leq {1 \over 2\tau}t(1-t)W_2^2(\mu_0, \mu_1)+ \Phi_\tau(\mu_t).
\]  
This proves (i).

(ii)  We have 
\[
0 \leq \Phi_\tau(\mu) \leq {1 \over 2 \tau}W_2^2(\mu, \nu_\ast) +\Phi(\nu_\ast).
\] 
This, together with the triangle inequality  
\[
(W_2(\mu, \nu_\ast))^2 \leq  \Bigl(W_2(\mu, \delta_0)+W_2(\delta_0, \nu_\ast) \Bigr)^2 \leq 4M_2(\mu)+4M_2(\nu_\ast),
\] 
give (ii).  

(iii) Let $\mu_0, \mu \in \mathcal P_2(\mathbb R^D)$. We have 
\begin{equation}\label{eq:jan03.2015.3}
\Phi_\tau(\mu) \leq \Phi(\mu_0^\tau)+{W_2^2(\mu, \mu_0^\tau) \over 2 \tau}= \Phi_\tau(\mu_0)- {W_2^2(\mu_0, \mu_0^\tau) \over 2 \tau}+{W_2^2(\mu, \mu_0^\tau) \over 2 \tau}.
\end{equation}
By Theorem 7.3.2 \cite{ambrosio2008gradient}, $\psi:= -1/2 W_2^2(\cdot, \mu_0^\tau)$ is $(-1)$--convex along geodesics. Since ${\bf id} - G_{\mu_0}^{\mu^\tau_0}  \in  \ubar{\partial} \psi (\mu_0)$, by Theorem 10.3.6 \cite{ambrosio2008gradient}, we have 
\[ 
\psi(\mu) \geq \psi(\mu_0)+ \int_{\mathbb R^D \times \mathbb R^D} \Bigl
\langle G_{\mu_0}^{\mu^\tau_0}(w)- w ;z-w  \Bigr\rangle \bar G(dw, dz)- {1 \over 2 }W_2^2(\mu, \mu_0).
\] 
This, along with (\ref{eq:jan03.2015.3}), yield (iii). 

We use (i), (iii) and Theorem 10.3.6 \cite{ambrosio2008gradient} to obtain (iv). \endproof

 \begin{remark} \label{re:moreau-sub}   Let $\mu \in \mathcal P_2(\mathbb R^D)$ and $G \in \Gamma_{o}(\mu, \mu^\tau)$. Furthermore, let $G_{\mu^\tau}^{\mu}$ be  the barycentric projection of $G$ based at  $\mu^\tau.$   
 \begin{enumerate}
\item[(i)]  We have 
\[
{ G_{\mu^\tau}^{\mu}-{\bf id} \over \tau} \in \ubar{\partial} \Phi(\mu^\tau), \quad  
{{\bf id}-G^{\mu^\tau}_{\mu} \over \tau} \in \bar{\partial} \Phi_\tau(\mu)
\]
\item[(ii)] We have 
\[
{W_2^2(\mu^\tau, \mu) \over \tau^2} \geq \Big\| {{\bf id} - G^{\mu}_{\mu^\tau} \over  \tau} \Big\|^2_{\mu^\tau},  \quad 
 \Big\| {{\bf id} - G^{\mu^\tau}_{\mu} \over  \tau} \Big\|^2_{\mu}
\] 
\item[(iii)] If $R>0$ and $\mu_1, \mu_0 \in \mathcal P_2(\mathbb R^D)$ are such that $W_2(\mu_0, \delta_0) \leq R$ and $W_2(\mu_1, \delta_0) \leq R$, then  for a constant $ \bar C_{\tau,R}$ depending on $R$ and $\tau$ 
\[
|\Phi_\tau(\mu_1)-\Phi_\tau(\mu_0)| \leq \bar C_{\tau,R} W_2(\mu_1, \mu_0)
\]
\end{enumerate}
\end{remark}
\proof{} The first claim in (i) can be derived from Lemma  10.3.4 \cite{ambrosio2008gradient} while the second claim is Lemma \ref{le:moreau-convex} (iv). The inequalities in (ii) are consequences of Jensen's inequality.  

(iii) Assume $R>0$ and $\mu_1, \mu_0 \in \mathcal P_2(\mathbb R^D)$ are such that $M_2(\mu_0), M_2(\mu_1) \leq R$. Without loss of generality, we may assume that $0 \leq \Phi_\tau(\mu_1)-\Phi_\tau(\mu_0)$. Let $\bar G \in \Gamma_{o}(\mu_0, \mu_1)$, let   $G \in \Gamma_{o}(\mu_0, \mu_0^\tau)$ and denote by $G_{\mu_0}^{\mu_0^\tau}$ the barycentric projection of $G.$ By Lemma \ref{le:moreau-convex} (iii) 
\[
|\Phi_\tau(\mu_1)-\Phi_\tau(\mu_0)| \leq  
\int_{\mathbb R^D \times \mathbb R^D} \Bigl \langle  {w - G_{\mu_0}^{\mu^\tau_0}(w) \over \tau} ; z-w \Bigr\rangle \bar G(dw, dz)+ {1 \over 2 \tau}W_2^2(\mu_1, \mu_0)
\] 
and hence, by H\"older's inequality  
\[
|\Phi_\tau(\mu_1)-\Phi_\tau(\mu_0)| \leq \biggl\|  {w - G_{\mu_0}^{\mu^\tau_0}(w) \over \tau}\biggr\|_{\mu_0}W_2(\mu_1, \mu_0) + {W_2^2(\mu_1, \mu_0) \over 2 \tau}.
\] 
We then use (ii) to obtain 
\[
|\Phi_\tau(\mu_1)-\Phi_\tau(\mu_0)| \leq  W_2(\mu_1, \mu_0) \biggl( {W_2(\mu^\tau_0, \mu_0) \over \tau^2}  + {W_2(\mu_1, \mu_0) \over 2 \tau}\biggr) \leq W_2(\mu_1, \mu_0) \biggl( \sqrt{ {\Phi_\tau(\mu_0) \over  \tau}}  + {W_2(\mu_1, \mu_0) \over 2 \tau}\biggr) .
\]
We use Lemma \ref{le:moreau-convex} (ii) to conclude. \endproof

\begin{remark}\label{re:moreau-fisher2} Assume $(\mu_k)_k \subset \mathcal P_2(\mathbb R^D)$ converges narrowly to $\mu \in \mathcal P_2(\mathbb R^D)$. If there exists $\eta \in \mathcal P_2(\mathbb R^D)$  such that 
\begin{equation}\label{eq:dec30.2015.6}
\lim_{k \rightarrow \infty} W_2(\mu_k, \eta)=W_2(\mu, \eta) 
\end{equation}
 then $(\mu_k)_k$ converges in the Wasserstein metric to $\mu$; this is by now a standard result.  
\end{remark}

\begin{lemma}\label{le:moreau-fisher2} Suppose $(\mu_n)_n$ is a bounded sequence in $\mathcal P_2(\mathbb R^D)$ that converges narrowly to $\mu \in \mathcal P_2(\mathbb R^D)$. Let $\mu_n^\tau \in J^\Phi_\tau(u_n)$ and let  $G_n \in \Gamma_{o}(\mu_n, \mu_n^\tau)$. 
 \begin{enumerate} 
\item[(i)] Up to  a subsequence, $(\mu_n^\tau)_n \subset \mathcal P_2(\mathbb R^D)$ converges in the Wasserstein metric to some $\mu^\tau.$ Furthermore, a subsequence of $(G_n)_n$ obtained from a second extraction has itself a subsequence which converges narrowly to some $G \in \Gamma_{o}(\mu, \mu^\tau)$.  
\item[(ii)] If $J_\tau^\Phi(\mu)=\{\mu^\tau\}$, then  the whole sequence $(\mu_n^\tau)_n \subset \mathcal P_2(\mathbb R^D)$ converges in the Wasserstein metric to $\mu^\tau.$
\item[(iii)] If $J_\tau^\Phi(\mu)=\{\mu^\tau\}$ and $\Gamma_{o}(\mu, \mu^\tau)$ has a unique element $G$, then the whole sequence $(G_n)_n$  converges narrowly to  $G$. 
\end{enumerate}  
\end{lemma} 
\proof{}  (i) Assume $(\mu_n)_n \subset \mathcal P_2(\mathbb R^D)$ narrowly converges  to $\mu$. Since  $\Phi \geq 0$, we use Lemma \ref{le:moreau-convex} to conclude that   
\begin{equation}\label{eq:dec30.2015.2.5}
\sup_n W_2(\mu_n, \mu_n^\tau)<\infty \quad \hbox{and} \quad \sup_n \Phi(\mu_{n}^\tau)<\infty. 
\end{equation}  
This, together with the fact that $(\mu_n)_n $ is bounded in $\mathcal P_2(\mathbb R^D)$, imply that $(\mu_n^\tau)_n $ is bounded in $\mathcal P_2(\mathbb R^D).$ Consider a subsequence $(\mu_{n_k}^\tau)_k \subset \mathcal P_2(\mathbb R^D)$. Since bounded subsets of  $\mathcal P_2(\mathbb R^D)$ are tight (cf., e.g., Remark 5.1.5 \cite{ambrosio2008gradient}) we may assume without loss of generality that $(\mu_{n_k}^\tau)_k \subset \mathcal P_2(\mathbb R^D)$ converges narrowly to some $\bar \mu \in \mathcal P_2(\mathbb R^D)$. Because $(G_{n_k}^\tau)_k \subset \mathcal P(\mathbb R^D \times \mathbb R^D)$ is tight, extracting a subsequence if necessary, we may assume that $(G_{n_k}^\tau)_k$ converges narrowly to some $G.$ By the stability of optimal transport plans for the narrow convergence (cf., e.g., Proposition 7.1.3 \cite{ambrosio2008gradient}), $G \in \Gamma_{o}(\mu, \bar \mu)$ and 
\begin{equation}\label{eq:dec30.2015.4}
\liminf_{k \rightarrow \infty}  W_2(\mu_{n_k}, \mu_{n_k}^\tau) \geq W_2(\mu, \bar \mu).
\end{equation} 
The lower semicontinuity of $\Phi$ for the narrow convergence and the second inequality in (\ref{eq:dec30.2015.2.5}) allow us to assert that   
\begin{equation}\label{eq:dec30.2015.3}
\infty> \liminf_{k \rightarrow \infty} \Phi(\mu_{n_k}^\tau) \geq \Phi(\bar \mu).
\end{equation}  
If $\nu \in \mathcal P_2(\mathbb R^D)$ then 
\[
\Phi(\nu)+{W_2^2(\mu_{n_k}, \nu) \over 2 \tau} \geq \Phi(\mu_{n_k}^\tau) +{W_2^2(\mu_{n_k}, \mu_{n_k}^\tau) \over 2 \tau}.
\]
Therefore, by (\ref{eq:dec30.2015.4}) and (\ref{eq:dec30.2015.3})
\begin{equation}\label{eq:dec30.2015.5}
\Phi(\nu) +{W_2^2(\mu, \nu) \over 2 \tau} \geq  \Phi(\bar \mu) +{W_2^2(\mu, \bar \mu) \over 2 \tau}.
\end{equation} 
Hence, $\bar \mu \in J^\Phi_\tau(\mu)$.  Would the inequality in (\ref{eq:dec30.2015.4}) be strict, so would be the one in (\ref{eq:dec30.2015.5}), yielding a contradiction. Thus, 
\[ 
\lim_{k \rightarrow \infty}  W_2(\mu_{n_k}, \mu_{n_k}^\tau) = W_2(\mu, \bar \mu).
\]    
The identities  
\begin{eqnarray}
|W_2(\mu, \mu_{n_k}^\tau)-W_2(\mu, \bar \mu)| 
&=&
\Bigl|\bigl(W_2(\mu, \mu_{n_k}^\tau)-W_2(\mu_{n_k}^\tau, \mu_{n_k})\bigr) + 
\bigl(W_2(\mu_{n_k}^\tau, \mu_{n_k}) -W_2(\mu,  \bar \mu)\bigr) \Bigr| \nonumber\\
&\leq &
W_2(\mu, \mu_{n_k}) + \bigl|W_2(\mu_{n_k}^\tau, \mu_{n_k}) -W_2(\mu,  \bar \mu)  \bigr|
\nonumber
\end{eqnarray}
yield   
 \[
 \lim_{k \rightarrow \infty} W_2(\mu, \mu_{n_k}^\tau)= W_2(\mu, \bar \mu).
 \]
 We apply Remark \ref{re:moreau-fisher2} to conclude that $(\mu_{n_k}^\tau)_n \subset \mathcal P_2(\mathbb R^D)$ converges in the Wasserstein metric to $\bar \mu.$ 
 
(ii) By (i), if $\mu^\tau$ is unique, every subsequence of  $(\mu_n)_n $ admits itself a subsequence converging to $\mu^\tau.$ Hence, the whole sequence must converge to $\mu^\tau.$ 

(iii) As in (ii), we use (i) to conclude that if $\mu^\tau$ is unique and $G$ is the unique element of $\Gamma_{o}(\mu, \mu^\tau)$, then the whole sequence  $(G_n)_n $ must converge to $G.$ \endproof

%
%
%
\section{Functions on $\mathcal P_2(\mathbb R^{2d})$ depending only on first marginals }\label{sec:projection}  
To emphasize the difference between the spatial and velocity variables, we set 
\[M:=\mathbb R^d, \quad TM:=M \times \mathbb R^d, 
\] 
and use notation such as $x \in M$, $(x, a) \in M \times M$, $(x,v) \in TM$, and so forth.
 
Suppose  
\[ \Phi: \mathcal P_2(TM) \rightarrow (-\infty, \infty], \quad \phi: \mathcal P_2(M) \rightarrow (-\infty, \infty]
\] 
 are lower semicontinuous for the narrow convergence and  
\[ \Phi(\mu) = \phi(\pi^1_\# \mu) \qquad \forall \, \mu \in \mathcal P_2(TM).\] 

In this section we study the relation between the superdifferential of the Moreau--Yosida approximations $\Phi_\tau$ at $\mu \in \mathcal P_2(\mathbb R^{2D})$ and that of $\phi_\tau$ at $\pi^1_\# \mu \in \mathcal P_2(\mathbb R^{D}).$  The set 
\[
S:=\{(x,v,a,b) \in TM \times TM \; | \; v=b\}
\]
plays an important role in our study.

\begin{definition}\label{de:optimal} Let $\mu \in \mathcal P_2(TM)$, $\eta \in \mathcal P_2(M)$, $\pi^{1}_\# \mu=\varrho$ and let $ \gamma \in \Gamma_{o}(\varrho, \eta)$. Let $(\mu_x)_x$ be the disintegration of $\mu$ with respect to $\varrho$ in the sense that  
\[
\int_{TM} l(x,v) \mu(dx, dv)= \int_M \varrho(dx) \int_{\mathbb R^d} l(x,v) \mu_x(dv) \qquad \forall \; l \in C_b(TM).
\]
\begin{enumerate} 
\item[(i)] We define the Borel measure $G:=   G^{\mu, \gamma}$ on $TM \times TM$  by 
\begin{equation}\label{eq:jan01.2016.8}
\int_{TM \times TM} g(x,v,a,b) G(dx, dv, da, db)=\int_{M \times M}  \gamma(dx, da) \int_{\mathbb R^d} g(x,v,a,v) \mu_x(dv) \qquad \forall g \in C_c\bigl( TM \times TM\bigr)
\end{equation}  
\item[(ii)] We define the Borel measure $m^{\mu,  \gamma} $ by 
\begin{equation}\label{eq:jan01.2016.8late}
\int_{TM} g(a,b)  m^{\mu,  \gamma} (da, db)= \int_{M \times M} \gamma(dx, da) \int_{\mathbb R^d} g(a,b) \mu_x(db) \qquad \forall g \in C_c(TM).
\end{equation}
\end{enumerate}
\end{definition}

\begin{remark}\label{re:optimal2} Using the above notation, the following hold: 
\begin{enumerate} 
\item[(i)]  $G^{\mu, \gamma}$ is supported by the closed set $S$. 
\item[(ii)]  $ G^{\mu, \gamma} \in \Gamma_{o}(\mu, m^{\mu, \gamma})$. 
\item[(iii)]  $\pi^1_\#m^{\mu, \gamma} =\eta.$ 
\item[(iv)]  $W_2(\varrho, \eta)=W_2(\mu, m^{\mu, \gamma}).$ 
\end{enumerate}
\end{remark} 
\proof{} (i) Observe that 
\[
\int_{TM \times TM} |v-b|^2  G^{\mu, \gamma} (dx, dv, da, db)= \int_{M \times M}  \gamma(dx, da) \int_{\mathbb R^d} 0 \mu_x(dv)=0,
\] 
which proves that $G^{\mu, \gamma}$ is supported by the closed set $S.$ 

(ii) Let $g \in C_c(TM).$  We have 

\begin{eqnarray} 
\int_{TM \times TM} g(x,v) G^{\mu,  \gamma}(dx, dv, da, db) 
&=& \int_{M \times M}\gamma(dx, da) \int_{\mathbb R^d} g(x,v) \mu_x(dv)  \nonumber\\
&= &
\int_{M}\varrho(dx) \int_{\mathbb R^d} g(x,v) \mu_x(dv)\nonumber\\ 
&= & \int_{TM}\varrho(dx) g(x,v) \mu(dx, dv). \label{eq:jan01.2016.10}
\end{eqnarray}  

Similarly,  
\begin{eqnarray} 
\int_{TM \times TM} g(a,b)G^{\mu, \gamma}(dx, dv, da, db) 
&=& \int_{M \times M} \gamma(dx, da) \int_{\mathbb R^d} g(a,v) \mu_x(dv)  \nonumber\\
&=& \int_{M \times M} \gamma(dx, da) \int_{\mathbb R^d} g(a,b) \mu_x(db)  \nonumber\\
&= &
\int_{T M} g(a,b)  m^{\mu, \gamma}(da, db). \label{eq:jan01.2016.11}
\end{eqnarray}  
By (\ref{eq:jan01.2016.10}) and (\ref{eq:jan01.2016.11}), $G^{\mu,  \gamma} \in \Gamma(\mu, m^{\mu, \gamma}).$

To conclude that $G \in \Gamma_{o}(\varrho, \eta_0)$, it suffices to show that the support of $G$ is cyclically monotone (cf. e.g. Section 6.2.3 \cite{ambrosio2008gradient}). Let $\{(x_i,v_i,a_i,b_i)\}_{i=1}^n \subset {\rm spt\,} G$ and let $\sigma$ be a permutation of $n$ letters. By (i), $b_i=v_i$, and therefore, using the fact that $\{(x_i,a_i)\}_{i=1}^n \subset {\rm spt\,} \gamma$ and $\gamma \in \Gamma_{o}(\varrho, \nu_0)$ we conclude that 
\[
\sum_{i=1}^n |(x_i,v_i)-(a_i,b_i)|^2=  \sum_{i=1}^n |x_i-a_i|^2
\leq \sum_{i=1}^n |x_i-a_{\sigma(i)}|^2+  \sum_{i=1}^n |v_i -b_{\sigma(i)}|^2.
\] 
Equivalently, this means  
\[
\sum_{i=1}^n |(x_i,v_i)-(a_i,b_i)|^2
\leq  \sum_{i=1}^n |(x,v_i)-(a_{\sigma(i)},b_{\sigma(i)})|^2.
\]  
Thus, the support of  $G^{\mu, \eta, \gamma}$ is cyclically monotone, which concludes the proof of (ii).

(iii) Let $g \in C_c(M).$ We have 
\[
\int_{TM \times TM} g(a)\bar m( da, db)= \int_{M \times M}\bar \gamma(dx, da) \int_{\mathbb R^d} g(a) \mu_x(db)= 
\int_{M \times M}g(a) \bar \gamma(dx, da) =\int_{M}g(a) \eta(da) .
\] 
Thus $\pi^1_{\#} \bar m=\eta.$ 

(iv)  Using the fact that by (i) $G^{\mu, \gamma}$ is supported by $S$ and by (ii) it is optimal, we have 
\[ 
W_2^2(\mu, m^{\mu, \gamma})=\int_{TM \times TM}|(x,v)-(a,b)|^2 G^{\mu, \gamma}(dx, dv, da, db)
=\int_{TM \times TM} |x-a|^2 G(dx, dv, da, db).
\]
Since 
\[\pi^{1,3}_\# G^{\mu, \gamma}= \gamma \in \Gamma_{o}(\varrho, \eta),\] 
the previous identity becomes  $W_2^2(\mu, m^{\mu, \gamma})=W_2^2(\varrho, \eta).$ \endproof

\begin{lemma}\label{le:bound-on-W2} Let $\mu \in \mathcal P_2(TM)$ and let $\varrho, \eta \in \mathcal P_2(M)$ be such that $\pi^{1}_\# \mu=\varrho.$ 
\begin{enumerate} 
\item[(i)] We have 
\begin{equation}\label{eq:jan01.2016.7}
\inf_{m \in \mathcal P_2(TM)} \Bigl\{ W_2^2(\mu, m)\; | \; m \in \mathcal P_2(TM), \pi^1_\# m= \eta \Bigr\}= W_2^2(\varrho, \eta).
\end{equation}  
\item[(ii)]  If $\gamma \in \Gamma_{o}(\varrho, \eta)$, then $m^{\mu, \gamma}$ minimizes (\ref{eq:jan01.2016.7}).  
\item[(iii)] If $\bar m$ minimizes (\ref{eq:jan01.2016.7}) and $\bar G \in \Gamma_{o}(\mu, \bar m)$, then $\bar \gamma:=\pi^{1,3}_\# \bar G \in \Gamma_{o}(\varrho, \eta)$ and $\bar G$ is supported by $S.$ 
\item[(iv)] If $\varrho\ll\mathcal L^D$, then $\bar m= m^{\mu, \gamma}$ is the unique minimizer in (\ref{eq:jan01.2016.7}) and $\Gamma_{o}(\mu, \bar m)=\{G^{\mu, \gamma}\}.$
\end{enumerate}
\end{lemma}

\proof{}  Let $m \in \mathcal P_2(TM)$ be such that  $\pi^1_\# m= \eta$ and let $G \in \Gamma_{o}(\mu, m).$ Set $\bar \gamma:= \pi^{1,3}_\# G \in \Gamma(\varrho, \eta).$ We have 

\begin{eqnarray} 
W_2^2(\mu, m)&=& \int_{TM \times TM}  |(x,v)-(a, b)|^2 G(dx, dv, da, db) \nonumber\\ 
&=& \int_{TM \times TM} \bigl( |x-a|^2+ |v-b|^2\bigr) G(dx, dv, da, db)  \nonumber\\
&= &
\int_{M \times M}|x-a|^2\bar \gamma(dx, da)+ \int_{TM \times TM} |v-b|^2 G(dx, dv, da, db)
\nonumber\\ 
&\geq & W_2^2(\varrho, \eta). \label{eq:jan01.2016.9}
\end{eqnarray} 
Observe that the inequality in (\ref{eq:jan01.2016.9}) is strict unless $\bar \gamma \in \Gamma_{o}(\varrho, \eta)$ and $G$ is supported by $S.$  In light of Remark \ref{re:optimal2} and (\ref{eq:jan01.2016.9})   
\[
W_2^2(\mu, m) \geq W_2^2(\varrho, \eta)= W_2^2(\mu, m^{\mu, \eta, \gamma}).
\] 
Hence, we have established (i) and (ii).

(iii) From the previous result, if $\bar m$ is another minimizer in (\ref{eq:jan01.2016.7}) and $\bar G \in \Gamma_{o}(\mu, m)$, then $ \bar G$ must be supported by $S$ and we must have $\bar \gamma:=\pi^{1,3}_\# \bar G \in \Gamma_{o}(\varrho, \eta)$, otherwise the inequality in (\ref{eq:jan01.2016.9}) would be strict. 

(iv) Assume now that $\varrho\ll\mathcal L^D$ and let $u: \mathbb R^D \rightarrow (-\infty, \infty]$ be a lower semicontinuous convex function such that $({\bf id} \times \nabla u)_\# \varrho=\bar \gamma.$ The first of the following identities is due to (iii). If $g \in C_c^\infty(TM \times TM)$, then 
\begin{eqnarray} 
\int_{TM \times TM}g(x,v,a,b) \bar G(dx, dv, da, db) &=& \int_{TM \times TM}g(x,v,\nabla u(x) , v) \bar G(dx, dv, da, db) \nonumber\\ 
&=& \int_{TM} g(x,v,\nabla u(x) , v) \mu(dx, dv)  \nonumber\\ 
&=& \int_{M} \varrho(dx) \int_{\mathbb R^D} g(x,v,\nabla u(x) , v)\mu_x(dv)  \nonumber\\
&= & \int_{TM} \gamma(dx, da) \int_{\mathbb R^D} g(x,v,\nabla u(x) , v)\mu_x(dv)  \nonumber\\
&= & \int_{TM} \gamma(dx, da) \int_{\mathbb R^D} g(x,v,a , v)\mu_x(dv)  \nonumber\\
&=& \int_{TM \times TM}g(x,v,a,b) G^{\mu, \gamma}(dx, dv, da, db).
\label{eq:jan03.2016.1}
\end{eqnarray} \endproof

\begin{definition}\label{de:metric_slope} Let $(\mathcal S, {\rm dist})$ be a metric space and let $\phi: \mathcal S \rightarrow [-\infty, \infty].$ If $v \in D(\phi)$, we define the global (metric) slope of $\phi$ at $v$ to be
\[
|\partial \phi|(v)= \limsup_{w \rightarrow v}{(\phi(v)-\phi(w))^+ \over {\rm dist}(w,v)}.
\]
\end{definition}

\begin{lemma}\label{le:subdifferential2} Let $\mu \in \mathcal P_2(TM)$ and let $\pi^{1}_\# \mu=\varrho.$ We have 
\[
|\partial \Phi|(\mu) = |\partial \phi|(\varrho).
\]
\end{lemma}  
\proof{} Lemma \ref{le:bound-on-W2} implies not only the straightforward inequality $|\partial \Phi|(\mu) \leq |\partial \phi|(\varrho)$, but in fact, it implies  that $|\partial \Phi|(\mu) = |\partial \phi|(\varrho)$. \endproof

\begin{lemma}\label{le:subdifferential1} Let $\mu \in \mathcal P_2(TM)$, let $\pi^{1}_\# \mu=\varrho$ and let $(\mu_x)_x$ be the disintegration of $\mu$ with respect to $\varrho.$ 
\begin{enumerate} 
\item[(i)] We have $\bar \xi_1 \in \bar{\partial} \phi(\varrho)$ if   
\[\xi =\left(
\begin{array}{c}
\xi_1\\
\xi_2\\
\end{array}
\right) \in \bar{\partial} \Phi(\mu) \quad \hbox{and} \quad \bar \xi_1(x)= \int_{\mathbb R^d} \xi_1(x, v) \mu_x(dv).\]  
\item[(ii)]  We have 
\[
||\xi||_\mu\geq ||\xi_1||_\varrho
\]
and the inequality is strict unless $\xi_2 = 0$ $\mu-$a.e. and $\xi_1(x,v)$ is independent on $v.$ 
\item[(iii)] If $\bar \partial \Phi(\mu) \not =\emptyset $, then $||\nabla_\mu \Phi(\mu)||_\mu \geq  ||\nabla_\varrho \phi(\varrho)||_\varrho.$ 
\end{enumerate}
\end{lemma} 

\proof{} (i) Let $\eta \in \mathcal P_2(M)$ and let $\gamma \in \Gamma_{o}(\varrho, \eta).$ Suppose $\xi$ and $\xi_1$ are as above. By Remark \ref{re:optimal2}, $G^{\mu,  \gamma} \in \Gamma_{o}(\mu, m^{\mu, \gamma})$ and $\pi^1_\#m^{\mu, \gamma} =\eta.$ Thus, (setting $w=(x,v)$ and $z=(a,b)$)
\begin{equation}\label{eq:jan02.2016.1}
\phi(\eta)-\phi(\varrho)= \Phi(m^{\mu,  \gamma})-\Phi(\mu) \leq \int_{TM  \times TM} \Bigl \langle 
\xi(w) ; z-w \Bigr\rangle G^{\mu, \eta, \gamma} + o\Bigl(W_2(\mu, m^{\mu, \eta, \gamma}) \Bigr)
\end{equation} 
By Remark \ref{re:optimal2} (iv),   
\begin{equation}\label{eq:jan02.2016.2}
W_2(\mu, m^{\mu, \gamma})=W_2(\varrho, \eta).
\end{equation}
But  
\begin{equation}\label{eq:jan02.2016.3} 
\int_{TM  \times TM} \Bigl \langle \xi(w) ; z-w \Bigr\rangle G^{\mu, \eta, \gamma} = \int_{M \times M} \gamma(dx, da) \int_{\mathbb R^d} 
\Bigl \langle 
\xi(w) ; \left(
\begin{array}{c}
a-x\\
0\\
\end{array}
\right) \Bigr\rangle \mu_x(dv).
\end{equation} 
We combine (\ref{eq:jan02.2016.1}), (\ref{eq:jan02.2016.2}) and (\ref{eq:jan02.2016.3}) to conclude that 
\[
\phi(\eta)-\phi(\varrho) \leq \int_{M \times M} \langle \bar \xi_1(x); a-x \rangle \gamma(dx, da) + o\Bigl(W_2(\varrho, \eta) \Bigr),
\] 
which proves (i).

(ii) Note that  
\[
||\xi||^2_\mu= \int_{TM} (|\xi_1|^2 + |\xi_2|^2)\mu(dx, dv) \geq \int_{TM} |\xi_1|^2\mu(dx, dv)= \int_M \varrho(dx) \int_{\mathbb R^d} |\xi_1|^2 \mu_x(dv),
\] 
and equality holds if and only if $||\xi_2||_\mu=0.$ Hence, by Jensen's inequality 
 \[
 ||\xi||^2_\mu \geq \int_M \varrho(dx)\Bigl| \int_{\mathbb R^d} \xi_1 \mu_x(dv) \Bigr|^2= ||\xi_1||^2_\varrho.
 \]
 The inequality is strict unless for $\rho$ a.e. $x$ we have $\xi_1(x,v)=\bar \xi_1(x)$ for a.e. $v.$ 

(iii) Follows from (i) and (ii). \endproof

\begin{remark}\label{re:subdifferential2} Let $\mu \in \mathcal P_2(TM)$ and let $\varrho= \pi^1_\# \mu.$ Let $G \in \Gamma_{o}(\mu, \mu^\tau)$ and recall that $G_\mu^{\mu^\tau}$ is its barycentric projection onto $\mu$.  
\begin{enumerate}  
\item[(i)] We have $\Phi_\tau(\mu)=\phi_\tau(\varrho)$ and $\varrho^\tau:= \pi^1_\# \mu^\tau.$   
\item[(ii)] We have $\pi^2\bigl( G_\mu^{\mu^\tau}(x,v) \bigr) \equiv v$ $\mu$ a.e., with $\pi^2$ defined by $\pi^2(x,v)=v$ for $(x,v) \in TM.$ 
\end{enumerate}
\end{remark}  
\proof{} (i) follows from Lemma \ref{le:subdifferential1}. 

Assume $G \in \Gamma_{o}(\mu, \mu^\tau)$ and let $A \in C_c(TM, \mathbb R^d)$ be arbitrary. We exploit Lemma \ref{le:subdifferential1} (iii) which asserts that $G$ is supported by $S$ to obtain

\begin{eqnarray} 
\int_{TM} \langle v-\pi^2\bigl( G_\mu^{\mu^\tau}(x,v) \bigr); A(x,v) \rangle \mu(dx, dv) 
&=& \int_{TM \times TM}  \langle v- b; A(x,v) \rangle G(dx, dv, da, db) \nonumber\\ 
&=& 0.
\label{eq:jan03.2016.2}
\end{eqnarray} \endproof

\begin{prop}\label{pr:diffmoreau-fisher2} Assume that $D(\phi) \subset \mathcal P_2^r(M)$ and that $\phi$ is convex for the $L^1$--metric.  Let $\mu \in \mathcal P_2(TM)$ and $\varrho= \pi^1_\# \mu$ be such that $J_\tau^\phi(\varrho)$ contains a unique element, $\varrho^\tau$, which then belongs to  $\mathcal P_2^r(M)$. Denote by  $\gamma$ the unique element of $\Gamma_{o}(\varrho, \varrho^\tau)$   and let $G \in \Gamma_0(\mu, \mu^\tau).$
 \begin{enumerate} 
\item[(i)]   If $\xi \in \bar{\partial} \phi_\tau(\varrho)$ and $\gamma_{\varrho}^{\varrho^\tau}$ denotes the barycentric projection of $\gamma$ onto $\varrho$, then 
\[ 
\pi_{\varrho}(\xi) = {  {\bf id}-\gamma_{\varrho}^{\varrho^\tau} \over \tau}. 
\]  
\item[(ii)] Further assume that $\varrho \ll \mathcal L^d$ and let $u: M \rightarrow (-\infty, \infty]$ be a lower semicontinuous convex function such that $(\nabla u)_\# \varrho=\varrho^\tau.$ If $X \in \bar{\partial} \Phi_\tau(\mu) $, then  
\[
\pi_{\mu}(X) = {{\bf id}-G_{\mu}^{\mu^\tau} \over \tau}= 
\left(
\begin{array}{c}
 {{\bf id}-\nabla u\over \tau} \\ 
 \\ 
0\\
\end{array}
\right). 
\]  
\item[(iii)] As a consequence, if $\varrho \ll \mathcal L^D$, then 
\[ \nabla_\varrho \phi_\tau(\varrho)={  {\bf id} -\nabla u \over \tau}, \quad \hbox{and} \quad 
\nabla_\mu \Phi_\tau(\mu)=  
\left(
\begin{array}{c}
\nabla_\varrho \phi_\tau(\varrho) \\ \\
0\\
\end{array}
\right). 
\]  
Furthermore, $J^\Phi_\tau(\mu)=\{m^{\mu, \gamma } \}.$
\end{enumerate}  
\end{prop} 
\proof{} (i) Applying Lemma \ref{le:moreau-convex} to $\phi_\tau$,  we have $\bar{\partial} \phi_\tau(\varrho) \not =\emptyset.$ For $U \in C^\infty_c(M)$ and for $s \in \mathbb R,$  we define  
\[
{\bf g}_s:={\bf id} +s \nabla U, \quad \hbox{and} \quad \varrho_s:= {\bf g}_{s \, \#} \varrho. 
\]
Observe that for $|s|$ small enough, ${\bf g}_s$ is the gradient of a convex function and therefore, it is optimal among the maps that push $\varrho$ forward to $\varrho_s$, where optimality is measured against the cost $c(x,a)=|x-a|^2$ where $x, a \in M.$ Hence, 
\[
\beta_s:=({\bf id} \times {\bf g}_s)_\# \varrho \in \Gamma_{o}(\varrho, \varrho_s).
\]
By the fact that $\varrho_s^\tau \in J_\tau^\phi(\varrho_s)$    we have  
\begin{equation}\label{eq:dec.30.2015.17late}
\phi_\tau(\varrho_s) -\phi_\tau(\varrho) \geq {1 \over 2 \tau}\Bigl(W_2^2(\varrho_s, \varrho_s^\tau)- W_2^2(\varrho, \varrho_s^\tau)  \Bigr).
\end{equation}
By the fact that $\xi \in \bar{\partial} \phi_\tau(\varrho)$,  there exists  a function $\bar \epsilon: \mathbb R \rightarrow \mathbb R$ such that $\lim_{t \rightarrow 0} \bar \epsilon(t)=0$ and 
\[
\phi_\tau(\varrho_s) -\phi_\tau (\varrho) \leq W_2(\varrho, \varrho_s) \bar \epsilon \bigl(W_2(\varrho, \varrho_s) \bigr)+ \int_{M \times M} \langle  \xi(x);a-x \rangle \beta_s(dx, da). 
\]
This, together with (\ref{eq:dec.30.2015.17late}), imply 

\begin{equation}\label{eq:dec.30.2015.17} 
{1 \over 2 \tau}\Bigl(W_2^2(\varrho_s, \varrho_s^\tau)- W_2^2(\varrho, \varrho_s^\tau)  \Bigr) \leq 
W_2(\varrho, \varrho_s) \bar \epsilon \bigl(W_2(\varrho, \varrho_s) \bigr)+ \int_{M \times M} \langle  \xi(x);a-x \rangle \beta_s(dx, da).
\end{equation} 
Let $\gamma_s \in \Gamma_{o}(\varrho_s, \varrho_s^\tau)$ and define on $M \times M$ the Borel probability measure $\bar \gamma_s$  by 
\[
\int_{M \times M} l(x,a) \bar \gamma_s(dx, da)= \int_{M \times M} F({\bf g}_s^{-1}(a),y)  \gamma_s(da, dy) \qquad \forall \; l \in C_b(M \times M).
\] 
We have
\[
{\bf g}_s^{-1}(a)=a-s \nabla U(a) +{s^2 \over 2} \nabla^2 U(a) \nabla U(a)+ o(s^2)
\]
and $\bar \gamma_s \in \Gamma (\varrho, \varrho_s^\tau)$. Thus, 
\begin{eqnarray}
W_2^2(\varrho_s, \varrho_s^\tau)- W_2^2(\varrho, \varrho_s^\tau) & \geq & 
\int_{M \times M} |a-y|^2 \gamma_s(dy, da)- \int_{M \times M} |a-x|^2 \bar \gamma_s(dx, da)  \nonumber\\
&= & 
\int_{M \times M} \bigl(|a-y|^2- |a- {\bf g}_s^{-1}(y)|^2\bigr) \gamma_s(dy, da) \nonumber\\
&= & 2s \int_{M \times M} \langle y-a ; \nabla U(y)\rangle \gamma_s(dx, da) + o(s). 
\label{eq:dec.30.2015.18}
\end{eqnarray} 
Recall that  for $|s|$ small enough, $\beta_s \in \Gamma_o(\varrho, \varrho_s)$ and hence, 
\begin{equation}\label{eq:dec.30.2015.19} 
W_2^2(\varrho, \varrho_s) = \int_{M \times M} |x-y|^2 \beta_s(dx, dy)= || s \nabla U||_{\varrho}^2.
\end{equation} 
We combine (\ref{eq:dec.30.2015.17}), (\ref{eq:dec.30.2015.18}) and (\ref{eq:dec.30.2015.19}) to obtain  
\[
 {o(s) \over s} +\int_{M \times M} \Bigl \langle {y-a \over \tau} ; \nabla U(y)  \Bigr\rangle  \gamma_s(da, dy)
\leq  ||\nabla U||_{\mu_0} \bar \epsilon \bigl( ||s \nabla U||_{\mu_0} \bigr)+ 
  \int_{M } \langle  \xi(x);\nabla U(x) \rangle \varrho(dx).
\] 
Letting $s \rightarrow 0$ we conclude that 
\begin{equation}\label{eq:dec.30.2015.20}  
\liminf_{s \rightarrow 0^+}
\int_{M \times M} \Bigl \langle {y-a \over \tau} ; \nabla U(y) \Bigr\rangle  \gamma_s(da, dy) \leq 
  \int_{M } \langle  \xi(x);\nabla U(x) \rangle \varrho(dx).
\end{equation}   
Observe that 
\[
\sup_{|s| \leq 1}W_2^2(\varrho_s, \delta_0)\leq \sup_{|s| \leq 1} \int_{M} |x+s \nabla U(x)|^2 \varrho(dx)<\infty.
\] 
This, together with Lemma \ref{le:moreau-convex} (ii), imply 
\[
\sup_{|s| \leq 1}W_2^2(\varrho_s^\tau, \delta_0)<\infty.
\] 
Thus, 
\begin{equation}\label{eq:dec.30.2015.20-new}
\sup_{|s| \leq 1}W_2^2(\gamma_s, \delta_{(0,0)})<\infty. 
\end{equation}  
By Lemma \ref{le:moreau-fisher2}, as $s$ tends to $0$, $(\gamma_s)_s$ converges narrowly to the unique element $\gamma \in \Gamma_{o}(\varrho, \varrho^\tau).$ Since (\ref{eq:dec.30.2015.20-new}) holds and $|{a-x \over \tau} ; \nabla U(x)|$ grows at most linearly as $|x|$ and $|a|$ tend to $\infty,$ we conclude that  
\[
\liminf_{s \rightarrow 0}
\int_{M \times M} \Bigl\langle {y-a \over \tau} ; \nabla U(y) \Bigr \rangle \gamma_s(da, dy)= 
\int_{M \times M} \Bigl\langle {y-a \over \tau} ; \nabla U(x) \Bigr \rangle \gamma(da, dy).
\]
This, together with (\ref{eq:dec.30.2015.20}), yield 
\[
\int_{M \times M} \Bigl\langle {y-a \over \tau} ; \nabla U(y) \Bigr \rangle \gamma(da, dy) \leq \int_{M } \langle  \xi(x);\nabla U(x) \rangle \varrho(dx).
\] 
Replacing $U$ by $-U$ we conclude that 
\[
 \int_{M } \langle  \xi(x);\nabla U(x) \rangle \varrho(dx) =\int_{M \times M} \Bigl\langle {y-a \over \tau} ; \nabla U(y) \Bigr \rangle \gamma(da, dy) = \int_{M} \Bigl\langle { y-\gamma_{\varrho}^{\varrho^\tau}(y)\over \tau} ; \nabla U(y) \Bigr \rangle \varrho(dy).
\]  
As a consequence, 
\[ 
\pi_{\varrho}( \xi) =\pi_{\varrho}\biggl(    {{\bf id}-\gamma_{\varrho}^{\varrho^\tau}  \over \tau} \biggr)={{\bf id}-\gamma_{\varrho}^{\varrho^\tau}  \over \tau}, 
\] 
since by  Theorems 8.5.5 and 12.4.4 \cite{ambrosio2008gradient}, we know that $\gamma_{\varrho}^{\varrho^\tau} -{\bf id}  \in T_{\varrho} \mathcal P_2(M).$ 

(ii) Further assume that $\varrho\ll\mathcal L^d.$ Then as observed in Remark \ref{re:moreau-fisher}, $J^\phi_\tau=\{\varrho^\tau\}$ reduces to a single point such that $\varrho^\tau\ll\mathcal L^d.$ Thus, $\Gamma_{o}(\varrho, \varrho^\tau)=\{\gamma\}$ also reduces to a single point and $\gamma= (\id \times \nabla u)_\# \varrho$ for a lower semicontinuous convex function, $u: M \rightarrow (-\infty, \infty].$ By Lemma \ref{le:bound-on-W2}  
\[
J^\Phi_\tau(\mu)= \{m^{\mu, \gamma}\}.
\] 
That uniqueness result is all we need to repeat the same arguments as in (i) to conclude the first identity in (ii). Remark \ref{re:subdifferential2} asserts that $\pi^2\bigl( G_\mu^{\mu^\tau}(x,v) \bigr) \equiv v$ while by Lemma  \ref{le:bound-on-W2} 
\[
\Gamma_{o}(\mu, \bar m)=\{G^{\mu, \gamma}\}.
\] 
Thus, if $A \in C_c(TM)$ is arbitrary, denoting by $G_{\mu}^{\mu^\tau}$ the barycentric projection of $G^{\mu, \gamma}$ onto $\mu$, we have 
\[
\int_{TM} \langle A(x,v) ; \pi^1(G_{\mu}^{\mu^\tau}(x,v))\rangle \mu(dx, dv)  = 
 \int_{TM \times TM} \langle A(x,v) ; a \rangle G^{\mu, \gamma}(dx, dv, da, dv). 
\]
Using the fact that $\gamma= (\id \times \nabla u)_\# \varrho,$ we conclude that 
\begin{eqnarray}
\int_{TM} \langle A(x,v) ; \pi^1(G_{\mu}^{\mu^\tau}(x,v))\rangle \mu(dx, dv) &= & 
 \int_{TM \times TM} \langle A(x,v) ; \nabla u(x) \rangle G^{\mu, \gamma}(dx, dv, da, dv) \nonumber\\
&= & 
 \int_{TM} \langle A(x,v) ; \nabla u(x) \rangle \mu(dx, dv).  
\label{eq:jan04.2015.4}
\end{eqnarray} 
Therefore, 
\[
 \pi^1(G_{\mu}^{\mu^\tau}(x,v))=\nabla u(x)=\gamma_{\varrho}^{\varrho^\tau} \qquad \mu \; \hbox{a.e.}.
\]

In light of (i), $ ({\bf id}-\gamma_{\varrho}^{\varrho^\tau})/ \tau$ is the element of minimal norm in $\bar{\partial} \phi_\tau(\varrho)$; hence, the first identity in (iii) holds. Similarly, we use (ii) to obtain the second identity in (iii). Since $J^\phi_\tau(\varrho)$ contains only $\varrho^\tau,$ we use Lemma \ref{le:bound-on-W2} (iv) to conclude that $J^\Phi_\tau(\varrho)$ contains only $m^{\mu,  \gamma }$. \endproof

%
%
%
\section{Solutions to an approximate Hamiltonian systems in the periodic setting}\label{sec:existence-approximate}
To avoid  technical issues, in this section,  we shall study an approximative version of the kinetic Bohmian equation (\ref{eq:Vl_eqn}) on $\mathbb T^d \times \mathbb R^d$ instead of  $\mathbb R^d \times \mathbb R^d.$ In the sequel, we set 
\[M:=\mathbb T^D, \] 
and  fix a function  $V \in C^2(M).$  The function $\mathcal F,$ defined in (\ref{eq:Fisher}) (or equivalently in (\ref{eq:fisher3})) as $1/8$ times the Fisher information, will be used in this section. For $\mu \in \mathcal P_2(TM)$, we define the function 
\[
\mathcal H(\mu) = M^1_2(\mu) + \Phi(\mu)+ \mathcal V(\mu)
\]
where
\[
\mathcal V(\mu) \equiv \mathcal V(\pi^1_\# \mu):= \int_{TM} V(x) \mu(dx, dv), \quad \Phi(\mu):= \phi(\pi^1_\# \mu), \quad  M^1_2(\mu)= \int_{TM}{|v|^2 \over 2} \mu(dx, dv),
\]
and 
\[
\phi:=\mathcal F.
\]
Fix $\tau>0$ and recall that if $\varrho \ll\mathcal L^d$ we denote by $\varrho^\tau$ the unique measure satisfying 
\[
\phi_\tau(\varrho)=\phi(\varrho^\tau)+{W_2^2(\varrho, \varrho^\tau)\over 2 \tau}.
\] 
Similarly, Lemma \ref{le:bound-on-W2} ensures that there is a unique $\mu^\tau \in \mathcal P_2(TM)$ such that 
\[
\Phi_\tau(\mu)=\Phi(\mu^\tau)+{W_2^2(\mu, \mu^\tau)\over 2 \tau}.
\] 
We set 
\[
\mathcal H_\tau(\mu) = M^1_2(\mu) + \Phi_\tau(\mu)+ \mathcal V(\mu).
\]

\begin{lemma}\label{le:oct29.2016.1} Let $\mu \in \mathcal P_2^r(TM)$ and assume that $\varrho:= \pi^1_\# \mu \ll\mathcal L^D.$ Then, 
\begin{equation}\label{eq:oct28.2016.1} 
\nabla_\mu \mathcal H_\tau(\mu)(x,v) = 
\left(
\begin{array}{c}
\nabla V(x) + { {\bf t}_\varrho^{\varrho^\tau}(x) -x \over \tau}   \\ 
\\ 
v\\
\end{array}
\right)=: {\bf H}(x,v),
\end{equation} 
where ${\bf t}_\varrho^{\varrho^\tau}$ is the optimal map that pushes $\varrho$ forward to $\varrho^\tau$
\end{lemma} 

\proof{} 
By Proposition \ref{pr:diffmoreau-fisher2} ,  
\[ 
\nabla_\mu \Phi_\tau(\mu)(x,v) = 
\left(
\begin{array}{c}
 { {\bf t}_\varrho^{\varrho^\tau}(x) -x \over \tau}   \\  
0\\
\end{array}
\right).
\] 
Since 
\[
\nabla_\mu M^1_2(\mu) \equiv \left(
\begin{array}{c}
0 \\ 
v\\
\end{array}
\right),  \quad  
\nabla_\mu \mathcal V(\mu) \equiv \left(
\begin{array}{c}
\nabla V(x) \\ 
0\\
\end{array}
\right) \quad \forall \; (x,v) \in TM,
\] 
and 

\[
\nabla_\mu M^1_2(\mu)  \in \bar \partial M^1_2(\mu) \cap \ubar \partial M^1_2(\mu) \quad \hbox{and} \quad  
\nabla_\mu \mathcal V(\mu)  \in \bar \partial \mathcal V(\mu) \cap \ubar \partial \mathcal V(\mu)
\] 
we conclude that if $Z \in \bar \partial \mathcal H_\tau(\mu)$, then 
\[ 
Z-\nabla_\mu M^1_2(\mu)-\nabla_\mu \mathcal V \in \bar \partial \Phi_\tau(\mu). 
\] 
Furthermore, by Proposition \ref{pr:diffmoreau-fisher2}  
\[
 \nabla_\mu \Phi_\tau(\mu)=\pi_\mu(Z-\nabla_\mu M^1_2(\mu)-\nabla_\mu \mathcal V)= \pi_\mu(Z)-\nabla_\mu M^1_2(\mu)-\nabla_\mu \mathcal V.
\] 
In particular, setting $Z:= \nabla_\mu \mathcal H_\tau(\mu)$, we conclude the proof. \endproof

\begin{theorem}\label{th:existence1} Let $\mu_0= f_0 \mathcal L^{2D} \in \mathcal P_2^r(TM)$ and let $\tau>0.$ 
\begin{enumerate} 
\item[(i)]   There exists a path $t \rightarrow \bar \mu_t^\tau$ such that for each $T>0$ we have $\bar \mu^\tau \in AC_2\bigl(0,T; \mathcal P_2(TM)\bigr)$ and 
\[
\partial_t \bar \mu^\tau+ \nabla \cdot \Bigl( \bar \mu^\tau J \nabla_\mu \mathcal H_\tau(\bar \mu^\tau) \Bigr)=0 \qquad \mathcal D'\bigl((0,T) \times TM) \bigr).
\] 
\item[(ii)] We have $\bar \mu_t^\tau \ll\mathcal L^{2D}$ for all $t>0.$   
\item[(iii)] Given $r \rightarrow M_r \in (0,\infty)$ there exists $r \rightarrow L_r \in (0,\infty)$ such that 
\[ f_0 \leq M_r \; \hbox{on} \; B_r(0) \quad \implies {d\bar \mu_t^\tau \over d\mathcal L^{2D} } \leq L_r \; \hbox{on} \; B_r(0) \] 
\item[(iv)] Given $r \rightarrow m_r \in (0,\infty)$ there exists $r \rightarrow l_r \in (0,\infty)$ (depending on $\tau$) such that 
\[ f_0 \geq m_r  \; \hbox{on} \; B_r(0)  \quad \implies {d\bar \mu_t^\tau \over d\mathcal L^{2D} } \geq l_r.  \; \hbox{on} \; B_r(0) \]  
\item[(v)] We have $\mathcal H_\tau(\bar \mu_t^\tau)=\mathcal H_\tau(\mu_0).$  
\end{enumerate} 
\end{theorem} 

\proof{} {\bf 1.} Let $\mu_0 \in \mathcal P_2^r(TM)$ and set $\varrho_0:= \pi^1_\# \mu_0.$ Similarly, for any arbitrary $\mu \in \mathcal P_2^r(TM)$ we set $\varrho:= \pi^1_\# \mu.$ Recall that ${\bf t}_\varrho^{\varrho^\tau}$ is the optimal map that pushes $\varrho$ forward to $\varrho^\tau$. Since ${\bf t}_\varrho^{\varrho^\tau}: M \rightarrow M$ and $M$ is a bounded set, Lemma \ref{le:oct29.2016.1} supplies us with a constant $C$ depending on $\tau$, but independent of $\mu$,  such that 
\begin{equation}\label{eq:oct28.2016.2} 
\Bigl|   \nabla_\mu \mathcal H_\tau(\mu)(x,v) \Bigr| \leq C(|(x,v)|+1),  \quad \forall (x,v) \in TM.
\end{equation} 
This is referred to as assumption (H1) in \cite{ambrosio2008hamiltonian}. 

Assume $(\mu_n)_n \subset \mathcal P_2(TM)$ is a sequence of absolutely continuous measures which converges narrowly to   $\mu\ll \mathcal L^{2d}$. Then $(\mu_n)_n$ is bounded in $\mathcal P_2(TM)$ for the Wasserstein metric and $(\varrho_n):= (\pi^1_\# \mu_n)_n$ is a sequence of absolutely continuous measures that converges narrowly to $\varrho \ll \mathcal L^d$. Let $u_n: \mathbb R^d \rightarrow \mathbb R$ be convex functions such that $x \rightarrow u(x)-|x|^2/2$ is convex, $u_n(0)=0$ and $\nabla u_n= {\bf t}_{\varrho_n}^{\varrho_n^\tau}$. By Remark \ref{re:moreau-fisher} both $J^\phi_\tau(\varrho)=\{\varrho^\tau\}$ and $J^\Phi_\tau(\mu)=\{\mu^\tau \}$ are of cardinality $1.$ By Lemma \ref{le:moreau-fisher2}, $(\varrho_n^\tau)_n$ converges to $\varrho^\tau$. Since $M$ is a compact set, $(\nabla u_n)_n$ is uniformly bounded on $M$. We use the convexity of $u_n$ to conclude that $(\nabla u_n)_n$ is pre--compact in $L^p(M)$ for any $1\leq p<\infty.$ Any point of accumulation of  $(\nabla u_n)_n$ in $L^p(M),$ ${\bf t}$, is an optimal map for the Wasserstein metric, $W_2$, among the maps that push $\varrho$ forward to $\varrho^\tau$. Since such an optimal map is unique, we conclude that the whole sequence $(\nabla u_n)_n$ converges to ${\bf t}= {\bf t}_{\varrho}^{\varrho^\tau}.$ Using the expression of $\nabla_\mu \mathcal H_\tau(\mu_n)$ provided by Lemma \ref{le:oct29.2016.1} we conclude that $\bigl(\nabla_\mu \mathcal H_\tau(\mu_n)\bigr)_n$ converges almost everywhere to $\nabla_\mu \mathcal H_\tau(\mu).$ This is referred to as assumption (H2) in \cite{ambrosio2008hamiltonian}.  By (H1) and (H2) we obtain (i)--(iv). 

{\bf 2.}  For the conservation of the Hamiltonian, \cite{ambrosio2008hamiltonian} requires the Hamiltonian to be $\lambda$--convex. We now check that $\lambda$--concavity is sufficient as well.

By Remark \ref{re:moreau-sub}, $\Phi_\tau$ is Lipschitz on bounded subsets of $\mathcal P_2(TM)$. Since $\mathcal V$ and $M^1_2$ are also Lipschitz on bounded subsets of $\mathcal P_2(TM)$, so is $\mathcal H_\tau=\Phi_\tau+M^1_2+\mathcal V$. Fix $T>0.$ Since $\mu \in AC_2\bigl(0,T; \mathcal P_2(TM)\bigr)$, we conclude that $t \rightarrow \mathcal H(\bar \mu^\tau_t)$ is Lipschitz on $[0,T]$. To show that $\mathcal H(\bar \mu^\tau_t)$ is time independent, it suffices to show that its derivative vanishes almost everywhere. 

 Let $W$ be the velocity of minimal norm for the path $t \rightarrow \bar \mu^\tau_t$ provided by Theorem 8.3.1 \cite{ambrosio2008gradient}. Since both $W$ and  $J \nabla_\mu \mathcal H(\bar \mu^\tau)$ are velocities for $t \rightarrow \bar \mu^\tau_t$, we have 
\[
\nabla \cdot \Bigl(W- J \nabla_\mu \mathcal H(\bar \mu^\tau)\Bigr)=0 \qquad \mathcal D'\Bigl((0,T) \times TM \Bigr).
\] 
In other words 
\[
\int_0^T dt \int_{TM} \langle W- J \nabla_\mu \mathcal H(\bar \mu^\tau_t); \nabla F \rangle \bar \mu^\tau_t(dx, dv) =0 \qquad \forall \, F\in C_0^1\bigl( (0,T) \times TM \bigr) .
\] 
Choosing $F$ in the form $F(t,x,v)= A(t) B(x,v)$ and using a density argument, we conclude that for almost every $t \in (0,T)$ we have 
\[
\int_{TM} \langle W- J \nabla_\mu \mathcal H(\bar \mu^\tau_t); \nabla B \rangle \bar \mu^\tau_t (dx, dv) =0 \qquad \forall \, B\in C_0^1(TM) 
\] 
Thus, for almost every $t \in (0,T)$, $W_t$ is the orthogonal projection of $J \nabla_\mu \mathcal H(\bar \mu^\tau_t)$ onto the tangent space $T_{\bar \mu^\tau_t} \mathcal P_2(TM):$  
\[
W_t:= \pi_{\bar \mu^\tau_t} \Bigl( J \nabla_\mu \mathcal H(\bar \mu^\tau_t)\Bigr).
\] 
By (8.4.6) \cite{ambrosio2008gradient}, for almost every $t \in (0,T)$, if $t+h \in (0,T)$ and $G_h\in \Gamma_{o}(\bar \mu^\tau_t, \bar \mu^\tau_{t+h})$, then we have the following convergence in the $W_2$--metric: 
\begin{equation}\label{eq:jan06.2016.1}
\lim_{h \rightarrow 0} \Bigl( \bar \pi^1, {\bar \pi^2- \bar \pi^1 \over h} \Bigr)_\# G_h= ({\bf id} \times W_t)_\# \bar \mu^\tau_t.
\end{equation}
Here, 
\[
\bar \pi^1(w,z)=w, \quad \bar \pi^2(w,z)=z \quad \forall \, w:=(x,v), z:=(a, b) \in TM.
\] 
Denote by $|(\bar \mu^\tau_t)'|$ the metric derivative of $t \rightarrow \bar \mu^\tau_t$ (cf. e.g. Definition 1.1.1 \cite{ambrosio2008gradient}). By definition 
\[
\lim_{h \rightarrow 0}{W_2(\bar \mu^\tau_t, \bar \mu^\tau_{t+h}) \over h}=|(\bar \mu^\tau_t)'|(t)
\] 
for almost every $t \in (0,T)$. Hence, for these $t$, 
 \begin{equation}\label{eq:jan06.2016.3}
{W_2^2(\bar \mu^\tau_t, \bar \mu^\tau_{t+h}) \over h}=o(h),
\end{equation}  
where $o(h)$ depends on $t$.  Note that by Lemma \ref{le:moreau-convex} (ii),  $\Phi_\tau$ is $\tau^{-1}$--concave. Since the second derivatives of $(x,v) \rightarrow V(x)$ and that of $(x,v) \rightarrow |v|^2$ are bounded, we conclude that  there exists a constant $\bar C_\tau$ such that $\mathcal H_\tau$ is $\bar C_\tau$--concave. Thus, 
\[
\mathcal H_\tau(\bar \mu^\tau_{t+h}) -\mathcal H_\tau(\bar \mu^\tau_t) \leq \int_{TM  \times TM} 
\langle \nabla_\mu \mathcal H(\bar \mu^\tau_t)(w); z-w \rangle G_h(dw, dz)+\bar C_\tau W_2^2(\bar \mu^\tau_t, \bar \mu^\tau_{t+h}).
\] 
If $t$ is such that (\ref{eq:jan06.2016.1}) holds, since $|\langle \nabla_\mu \mathcal H_\tau(\bar \mu^\tau_t); z-w\rangle|$ grows at most quadratically, we conclude that  
\begin{equation}\label{eq:jan06.2016.7}
\mathcal H_\tau(\bar \mu^\tau_{t+h}) -\mathcal H_\tau(\bar \mu^\tau_t) \leq \int_{TM} 
h \langle \nabla_\mu \mathcal H_\tau(\bar \mu^\tau_t)(w);W_t(w) \rangle \bar \mu^\tau_t(dw)+\bar C_\tau W_2^2(\bar \mu^\tau_t, \bar \mu^\tau_{t+h})+ o(h).
\end{equation} 
We use the fact that $W_t$ is the projection of $J\nabla_\mu \mathcal H(\bar \mu^\tau_t)$ onto $T_{\bar \mu^\tau_t} \mathcal P_2(TM)$ to conclude that 
\[
\int_{TM} \langle \nabla_\mu \mathcal H(\bar \mu^\tau_t) (w); W_t(w) \rangle\bar \mu^\tau_t(dw) = 
\int_{TM} \langle \nabla_\mu \mathcal H(\bar \mu^\tau_t) (w); J\nabla_\mu \mathcal H(\bar \mu^\tau_t)(w) \rangle \bar \mu^\tau_t(dw)=0.
\] 
This, together with (\ref{eq:jan06.2016.3}) and (\ref{eq:jan06.2016.7}), imply  
\begin{equation}\label{eq:jan06.2016.5}
\mathcal H(\bar \mu^\tau_{t+h}) -\mathcal H(\bar \mu^\tau_t) \leq o(h). 
\end{equation}
The map $t \rightarrow \mathcal H(\bar \mu^\tau_t)$ is Lipschitz on $[0,T]$. Therefore, it is differentiable almost everywhere. If $t$ is a point of differentiability, using alternatively $h>0$ and $h<0$ in (\ref{eq:jan06.2016.5}), we conclude that 
\[
{d \over ds}\mathcal H(\bar \mu^\tau_s)|_{s=t}=0.
\] 
Since the derivative of the Lispchitz function $t \rightarrow \mathcal H(\bar \mu^\tau_t)$ vanishes almost everywhere, the function must be constant. \endproof 

\begin{remark} If we replace $\mathbb T^d$ by $\mathbb R^d$ then, because of Remark \ref{re:moreau-sub}, (H1') of \cite{ambrosio2008hamiltonian} holds. \cite{ambrosio2008hamiltonian} ensures that if (H2')  also holds, then there is a solution to our Hamiltonian system. The proof of (H2') requires some effort and this is why we worked on $\mathbb T^d.$ Note that the above arguments go through if we replace $\mathbb T^d$ by any open bounded set. 
\end{remark}

%
%
%
\section{Ingredients toward a convergence analysis in the periodic setting} 
Let $\mu_0=f_0 \mathcal L^{2d} \in \mathcal P_2^r(TM)$ and let $T>0.$  For $\tau>0$ we define $t \rightarrow \bar \mu^\tau_t \in \mathcal P_2^r(TM)$ as in Theorem \ref{th:existence1}. Write
\[
\bar \mu^\tau_t = \bar f^\tau_t \mathcal L^{2d}, \quad  \pi_{1 \, \#} \bar \mu^\tau_t=\bar \varrho^\tau_t \mathcal L^d, \quad  \bar f^\tau_t(x,v)= \bar \varrho^\tau_t(x) \bar F^\tau_t(x,v), \quad \hbox{with} \quad \int_{\mathbb R^d} \bar F^\tau_t(x,v) dv=1.
\]

\subsection{Continuity equation}\label{subsec:cont}
Since $f_0 \in L^1(TM)$, we apply de la Vall\'ee Poussin Theorem to $\{f_0\},$ a compact subset of $L^1(TM)$, to conclude that there exists a super linear convex function $\theta: [0,\infty) \rightarrow [0,\infty)$ such that $\theta(f_0) \in L^1(TM).$ We use Lemma 6.2 \cite{ambrosio2008hamiltonian} to conclude that  
\begin{equation}\label{eq:jan30.2017.0}
\sup_{t \in [0,T]} \int_{TM} \theta(\bar f_t) dxdv\leq \int_{TM} \theta(\bar f_0) dxdv<\infty. 
\end{equation}
We apply again de la Vall\'ee Poussin Theorem to conclude that $\{\bar f^\tau\; | \; \tau>0 \}$ is a compact subset of $L^1\bigl((0,T) \times TM \bigr).$

Recall that since $\pi_{1 \, \#} \bar \mu^\tau_t\ll \mathcal L^d$, $J^\phi_\tau(\bar \varrho_t)$ reduces to a single element $\varrho^\tau_t \mathcal L^d.$ We have  
\begin{equation}\label{eq:jan30.2017.1}
\phi_\tau(\bar \varrho^\tau_t )= \phi(\varrho^\tau_t )+{W_2^2(\varrho^\tau_t, \bar \varrho^\tau_t)\over 2 \tau}
\end{equation}
By Theorem \ref{th:existence1} (v) 
\begin{equation}\label{eq:jan30.2017.2}
\phi_\tau (\bar \varrho^\tau_t ) + \int_{M} V(x) \bar \varrho^\tau_t (x)dx+ {1 \over 2} \int_{TM} |v|^2 \bar \mu^\tau_t(dx, dv)= \mathcal H_\tau(\mu_0) \leq \mathcal H(\mu_0).
\end{equation}
By Proposition \ref{pr:diffmoreau-fisher2} 
\begin{equation}\label{eq:jan30.2017.3}
||\nabla_\varrho \phi_\tau (\bar \varrho^\tau_t ) ||_{\bar \varrho^\tau_t }={W_2(\varrho^\tau_t, \bar \varrho^\tau_t)\over  \tau}. 
\end{equation}
This, together with (\ref{eq:jan30.2017.2}), yield 
\begin{equation}\label{eq:jan30.2017.4}
{\tau \over 2} ||\nabla_\varrho \phi_\tau (\bar \varrho^\tau_t ) ||_{\bar \varrho^\tau_t }^2+\phi(\varrho^\tau_t ) \leq \mathcal H(\mu_0) +||V||_\infty.
\end{equation}

Define 
\[
\bar {\bf u}^\tau_t(x):= \int_{\mathbb R^d} v \bar F^\tau_t(x,v) dv.
\]

We use (\ref{eq:jan30.2017.0}) to deduce that up to a subsequence, $(\bar f^\tau)_\tau$ converges weakly to some $\bar f$ in $L^1\bigl( (0,1) \times TM \bigr)$.

\begin{prop}\label{pr:kinetic} The following hold:
 \begin{enumerate} 
\item[(i)]  $ \bar \varrho^\tau \in AC_2\bigl(0,T; \mathcal P(M) \bigr).$ 
\item[(ii)] 
\[
{1 \over 2}  \int_{M} |\bar {\bf u}^\tau_t(x)|^2 \bar \varrho^\tau_t(x)dx \leq \mathcal H(\mu_0) +||V||_\infty.
\] 
\item[(iii)] 
\[
\partial_t \bar \varrho^\tau + \nabla \cdot (\bar \varrho^\tau \bar {\bf u}^\tau)=0 \qquad \mathcal D'\bigl((0,T) \times M \bigr).
\] 
\end{enumerate}  
\end{prop} 
\proof{} (i) We use that $\pi_1$ is a contraction of $\bigl(\mathcal P_2(TM), W_2\bigr)$ into $\bigl(\mathcal P_2(M), W_2\bigr)$ and use the fact that $\bar \mu^\tau \in AC_2\bigl(0, T; \mathcal P_2(TM) \bigr)$ to conclude the proof of (i).

(ii) We use Jensen's inequality to deduce that 
\[
\int_{M} |\bar {\bf u}^\tau_t(x)|^2 \bar \varrho^\tau_t(x)dx  \leq \int_{TM} |v|^2 \bar \mu^\tau_t(dx, dv),
\]
which, together with  (\ref{eq:jan30.2017.2}), yield (ii). 

(iii) The differential equation in Theorem \ref{th:existence1} yields (iii). \endproof

\subsection{Convergence in $\mathcal P(M)$} The goal of this subsection is to establish some convergence results. We prove that the paths $\bar{\varrho}^{\tau_n}$ and ${\varrho}^{\tau_n}$ converge to the same limit $\bar \varrho.$ Setting 
\[
\bar \nu^{\tau_n}:= \mathcal L^1_{(0,T)} \otimes\bar{\mu}^{\tau_n}
\]
we show that for the narrow convergence topology, $\bigl(\nu^{\tau_n}\bigr)_n$ contains points of accumulation  of the form $\mathcal L^1_{(0,T)} \otimes\bar \mu_{t}$ where $\bar \varrho_{t}\mathcal{L}^{d}$ is the projection of $\bar{\mu}_{t}$ onto $M.$

\begin{prop}\label{pr:convergence1} There exists a sequence $(\tau_n)_n$ decreasing to $0$ such that the following hold:
\begin{enumerate}
\item[(i)] For any $t \in (0,T),$ $\bigl(\overline{\varrho_t}^{\tau_n} \bigr)_n$ converges in $\mathcal P(M)$ to $\bar{\varrho_t} $. 
\item[(ii)] For any $t \in (0,T),$  $\bigl(\bar{\varrho}_{t}^{\tau_n} \bigr)_n$ converges in $\mathcal P(M)$ to $\bar{\varrho_t} $
\item[(iii)]  We have $\sup_{t \in (0,T)}\phi\left(\varrho_{t}\right)<\infty.$
\item[(iv)]  $\bigl(\nu^{\tau_n}\bigr)_n$ converges narrowly on $[0,T] \times TM$  to 
some $\nu=\mathcal L^1_{(0,T)} \otimes\bar \mu_{t}$.
\item[(v)] We have  $\bar \mu_{t}\left(TM\right)=1$ for $\mathcal{L}^{1}-$ a.e. $t\in\left(0,T\right)$.
\item[(vi)] We have $\pi_{\#}^{1}\bar{\mu}_{t}=\bar \varrho_{t}\mathcal{L}^{d}$ for $\mathcal{L}^{1}-$a.e. $t \in (0,T)$.
\end{enumerate}
\end{prop}

\proof{} Recall that $\mathcal{H}_{\tau} \leq \mathcal{H}$. Therefore, using Theorem \ref{th:existence1} (v) we have 
\begin{equation}\label{eq:feb23.2017.1}
\mathcal{H}_{\tau}\left(\bar{\mu}_{t}^{\tau}\right)=\mathcal{H}_{\tau}\left(\mu_{0}\right)\leq\mathcal{H}\left(\mu_{0}\right). 
\end{equation}

(i) By Proposition \ref{pr:kinetic}, $ \left|\left(\bar{\varrho}^{\tau}\right)'\right|^{2}\leq 2\mathcal{H}_{\tau}\left(\mu_{0}\right)+2\left\Vert V\right\Vert _{\infty}$ except maybe on a set of null measure. Thus, 
\[
W_{2}\left(\bar{\varrho}^{\tau}_t,\bar{\varrho}^{\tau}_s\right)\leq\int_{s}^{t}\left|\left(\bar{\varrho}^{\tau}\right)'\right|\left(l\right)dl\leq\left|t-s\right|\sqrt{2\left(\mathcal{H}\left(\mu_{0}\right)+\left\Vert V\right\Vert _{\infty}\right)}.
\]
Now we can apply the Ascoli-Arzela theorem (see Proposition 3.3.1 \cite{ambrosio2008gradient})
to get (i).

(ii) We exploit (\ref{eq:feb23.2017.1}) to get 
\begin{equation}\label{eq:feb23.2017.2}
\left\Vert V\right\Vert _{\infty}+\mathcal{H}\left(\mu_{0}\right)\geq\phi_{\tau}\left(\bar{\varrho}_{t}^{\tau}\right)=\frac{W_{2}^{2}\left(\bar{\varrho}_{t}^{\tau},\varrho_{t}^{\tau}\right)}{2\tau}+\phi\left(\varrho_{t}^{\tau}\right).
\end{equation}
Hence,
\[
W_{2}^{2}\left(\bar{\varrho}_{t}^{\tau},\varrho_{t}^{\tau}\right)\leq2\tau\left(\mathcal{H}\left(\mu_{0}\right)+\left\Vert V\right\Vert _{\infty}\right),
\]
which, together with (i), yield (ii).

(iii) We use (\ref{eq:feb23.2017.2}) and the fact that $\phi$ is lower semicontinuous for the narrow convergence to conclude that 
\[
\phi\left(\varrho_{t}\right)\leq\left\Vert V\right\Vert _{\infty}+\mathcal{H}\left(\mu_{0}\right).
\]

(iv) By (\ref{eq:jan30.2017.2})   
\begin{eqnarray}
\int_{0}^{T}dt\int_{TM}\left(1+\left|x\right|^{2}+\left|v\right|^{2}\right)\bar \mu_{t}^{\tau}\left(dx,dv\right) &= & 
\int_{0}^{T}\left(1+\int_{M}\left|x\right|^{2}\bar{\varrho}_{t}^{\tau}\left(dx\right)\right)dt+\int_{0}^{1}\int_{TM}\left|v\right|^{2}\bar \mu_{t}^{\tau}\left(dx,dv\right) \nonumber\\
&\leq & 
1+\left(\mathrm{diam}M\right)^{2}+2\left(\mathcal{H}\left(\mu_{0}\right)+\left\Vert V\right\Vert _{\infty}\right).  \nonumber\
\end{eqnarray} 
Hence, $\bigl(\bar \nu^{\tau_n}\bigr)_n$ is pre--compact for the narrow convergence. Extracting a subsequence if necessary, we obtain a Borel measure $\nu$ on $[0,1] \times TM$ such that $\bigl(\bar \nu^{\tau_n}\bigr)_n$ converges narrowly on $[0,T] \times TM$ to $\nu$. Since the projection of $\mathcal L^1_{(0,T)} \bar{\mu}_{t}^{\tau_{n}}$ onto $[0,T]$ is less than $1$, the same is true for the projection of $\nu$ (cf. e.g. Theorem 2.28 \cite{ambrosio2000functions}). This concludes the proof of (iv). 

(v) Let $\varphi\in C_{b}\left(\left[0,1\right]\right)$. Note that 
\[
\int_{0}^{1}\varphi\left(t\right)dt=  \underset{n}{\lim}\int_{0}^{1}\varphi\left(t\right)dt\int_{TM}\mu_{t}^{\tau_{n}}\left(dx,dv\right)= 
\underset{n}{\lim}\int_{[0,T] \times TM}\varphi\left(t\right)  \nu^{\tau_{n}}\left(dt, dx,dv\right).
\] 
We use (iv) to deduce that 
\[
\int_{0}^{1}\varphi\left(t\right)dt=\int_{0}^{1}\varphi\left(t\right)dt\int_{TM}\bar \mu_{t}\left(dx,dv\right).
\]
Since $t \rightarrow \int_{TM}\bar \mu_{t}\left(dx,dv\right)$ belongs to $L^{1}\left(0,1\right)$, (v) follows. 

(vi) Let $\varphi\in C_{b}\left(\left[0,1\right]\right)$ and $\psi\in C_{b}\left(M\right)$. We first use (i) and then use (v) to obtain 
\[
\int_{0}^{T}\varphi\left(t\right)dt\int_{M}\bar \varrho_{t}\left(x\right)\psi\left(x\right)dx = 
 \underset{\tau_{n}\rightarrow0}{\lim}\int_{0}^{T}\varphi\left(t\right)dt\int_{M}\bar{\varrho}_{t}^{\tau_{n}}\left(x\right)\psi\left(x\right)dx= 
 \underset{n}{\lim}\int_{[0,T] \times TM}\varphi\left(t\right)  \nu^{\tau_{n}}\left(dt, dx,dv\right).
\]
Thus by (iv), 
\[
\int_{0}^{T}\varphi\left(t\right)dt\int_{M}\bar \varrho_{t}\left(x\right)\psi\left(x\right)dx = \int_{0}^{T}\varphi\left(t\right)dt\int_{TM}\psi\left(x\right)\bar{\mu}_{t} \left(dx,dv\right),
\]
which means that 
\[
\int_{TM}\psi\left(x\right)\bar{\mu}_{t}\left(dx,dv\right)=\int_{M}\bar \varrho_{t}\left(x\right)\psi\left(x\right)dx.
\] 
Since $\psi\in C_{b}\left(M\right)$ is arbitrary, we conclude the proof of (vi). \endproof

\subsection{Momentum equations for approximate  solutions} Recall that according to Section \ref{sec:existence-approximate},  if  ${\bf t}^\tau$ is the unique gradient of a lower semicontinuous convex function such that  ${\bf t}^\tau_\# \varrho_t^\tau= \bar \varrho_t^\tau$, then 
\[
\xi^\tau:={{\bf t}^\tau - \id \over \tau} \in \ubar \partial \phi(\varrho_t^\tau) \quad \hbox{and} \quad  {\id- ({\bf t}^\tau)^{-1}  \over \tau} \in \bar \partial \phi(\bar \varrho_t^\tau)
\] 
Thus, by Proposition \ref{pr:diffmoreau-fisher2}, the Wasserstein gradient of $\phi_\tau$, $\bar \varrho_t^\tau$, and $\xi^\tau$ satisfy  the relation 
\begin{equation}\label{eq:gradient-fisher}
\xi^\tau=\nabla_\varrho \phi_\tau(\bar \varrho_t^\tau)  \circ {\bf t}^\tau.
\end{equation} 
Using  $\bar F^\tau$ as introduced at the beginning of the current section,  we define the averages 
\[
\widehat{ v \otimes v}^{\bar \mu^\tau}(t,x) = \int_{\mathbb R^d}  v \otimes v \bar F_t^\tau(x,v) dv.
\] 
\begin{definition}\label{de:weak-dist} Let $\varrho \in AC_2(0,T; \mathcal P(M))$. Moreover, let $\mu \in AC_2(0,T; \mathcal P(M))$ be such that $\varrho_t$ is the projection of $\mu_t$ on $M$ and set 
\[
\widehat{ v \otimes v}^{ \mu}:= \int_{\mathbb R^d}  v \otimes v  F_t(x,dv), 
\]
where $( F_t(x,\cdot))_x$ is the disintegration of $\mu_t.$ Assume that $\xi:(0,T) \times M \rightarrow \mathbb R^d$ is a Borel vector field such that $\xi_t \in L^2(\varrho_t)$ for $\mathcal L^1-$a.e. $t \in (0,1).$ We say that $(\varrho, u, \widehat{ v \otimes v}^{\mu}, \xi)$ satisfies the momentum equation 
\begin{equation}\label{eq:momentum}
\partial_t (\varrho u)+ \nabla \cdot \bigl(\varrho \widehat{ v \otimes v}^{\mu}\bigr) =-\varrho(\nabla V +\xi)
\end{equation}  
in the sense of distribution if 
\[
\int_0^T dt \int_M \Bigl(\partial_t A + \widehat{ v \otimes v}^{ \mu} \nabla A \Bigr)\varrho_t(dx)= \int_0^T dt \int_M \langle A; \nabla V + \xi\rangle \varrho_t(dx),
\]
for all $A \in C_c^\infty((0,T) \times M; \mathbb R^d).$
\end{definition}

\begin{remark} \label{rem:error-fisher} The following hold:
\begin{enumerate} 
\item[(i)] If $\varrho$ belongs to the appropriate Sobolev space, then it is smooth enough such that we can write the Wasserstein gradient of $\phi$ at $\varrho$ as 
\[
\nabla_\varrho \phi(\varrho)=- {1\over 2}\nabla \Bigl( {\triangle \sqrt \varrho \over \sqrt \varrho} \Bigr).
\]
Therefore (cf., e.g., \cite{markowich2010bohmian}), 
\begin{equation}\label{eq:gradient-fisher2}
\varrho \nabla_\varrho \phi(\varrho)={1\over 2} \nabla (\triangle \varrho) - {\rm div} \Bigl( \nabla \sqrt \varrho \otimes  \nabla \sqrt \varrho \Bigr). 
\end{equation} 
\item[(ii)] Since $\phi$ is the Fisher information up to a multiplicative constant and $J_\tau^\phi(\varrho_t^\tau)=\{\varrho_t^\tau\},$ by Lemma 10.1.2 \cite{ambrosio2008gradient},  
$\xi^\tau$ is in the strong subdifferential of $\phi.$ By Corollary 5.8 \cite{gianazza2009wasserstein} 
\begin{equation}\label{eq:estimate-fisher}
\sqrt  \varrho_t^\tau \in W^{2,2}(M), \qquad 
\end{equation} 
\item[(iii)] If $A\in C^1\bigl((0,T) \times M, \mathbb R^d \bigr)$, then we can apply Corollary 5.8 \cite{gianazza2009wasserstein}  to deduce that (\ref{eq:gradient-fisher2}) holds for $\varrho= \varrho^\tau$ in the sense that 
\begin{equation}\label{eq:gradient-fisher2.1}
\int_0^T dt \int_M \langle A; \xi^\tau \rangle \varrho_t^\tau(x)dx = \int_0^T dt \int_M  \Bigl( -{1\over 2} (\nabla \cdot A) \; \triangle  \varrho_t^\tau + \langle \nabla A; \nabla \sqrt  \varrho_t^\tau \otimes  \nabla \sqrt  \varrho_t^\tau \rangle \Bigr)dx.
\end{equation} 
\item[(iv)] Observe that Proposition \ref{pr:convergence1} (iii) alone ensures that, for the limiting measures, we have $\sqrt {\varrho_t^\tau} \in W^{1,2}$ and therefore, the expression on the right-hand side of (\ref{eq:gradient-fisher2.1}) continues to make sense for the limiting densities $\bar \varrho$ obtained in Proposition \ref{pr:convergence1}; it  can be written as 
\[
 \int_0^T dt \int_M  \Bigl( -{1\over 2} \triangle (\nabla \cdot A) \;   \bar \varrho_t + \langle \nabla A; \nabla \sqrt{\bar \varrho_t} \otimes  \nabla \sqrt {\bar \varrho_t} \rangle \Bigr)dx.
\] 
\end{enumerate} 
\end{remark}

\hfill\break

For any vector valued Borel field, $\xi$, on $M$ of null average, we  define the norm   
\[
||\xi||_{-1}=\sup_{A \in C_c^\infty(M; \mathbb R^d)} \biggl\{\int_M \langle A; \xi(dx) \rangle \; | \;\;\; ||\nabla A||_\infty \leq 1\biggr\}.
\]

\begin{theorem} Using the notation of Subsection \ref{subsec:cont}, the following hold: 
\begin{enumerate} 
\item[(i)] $\bigl(\bar \varrho^\tau, \bar u^\tau, \widehat{ v \otimes v}^{\mu}, \nabla_\varrho \phi_\tau(\bar \varrho^\tau) \bigr)$ satisfies the momentum equation (\ref{eq:momentum}) in the sense of distributions.
\item[(ii)] In the sense of distributions, as given by Definition \ref{de:weak-dist} and  (\ref{eq:gradient-fisher2.1}), 
\[
\partial_t (\bar \varrho^\tau \bar u^\tau)+ \nabla \cdot \bigl(\bar \varrho^\tau \widehat{ v \otimes v}^{\bar \mu^\tau}\bigr) =-\bar \varrho^\tau \nabla V + \nabla \Bigl( {1\over 2} \nabla (\triangle \varrho^\tau) - {\rm div} \bigl( \nabla \sqrt \varrho^\tau \otimes  \nabla \sqrt \varrho^\tau \bigr) \Bigr) +\vec 0^\tau,
\]
where 
\[
\vec 0^\tau:=  \bar \varrho^\tau \nabla_\varrho \phi(\bar \varrho^\tau)- \varrho^\tau \xi^\tau.
\]
\item[(iii)] Further assume that there exists a sequence $(\tau_n)_n$ decreasing to $0$ such that  for $\mathcal L^1$ a.e. $t \in (0,T)$ we have 
\begin{equation}\label{eq:lebesgued}
\lim_{n \rightarrow \infty} \phi_{\tau_n}(\bar \varrho_t^{\tau_n}) - \phi(\varrho_t^{\tau_n}) =0.
\end{equation}  
Then, for any $p \in [1,\infty)$ we have 
\[
\lim_{n \rightarrow \infty} \int_0^T ||\vec 0^{\tau_n}_t  ||_{-1}^p dt =0.
\]
\end{enumerate} 
\end{theorem} 
\proof{} (i) By Theorem \ref{th:existence1} for any $L \in C_c^\infty\bigl((0,T) \times TM \bigr)$ we have 
\begin{equation}\label{eq:distribution1}
\int_{0}^{T}    dt \int_{TM}  (\partial_t L + \langle v; \nabla_x L\rangle)\bar \mu^\tau_t(dx, dv)=\int_{0}^{T}    dt \int_{TM} \langle \nabla_v L; V+ \nabla_\varrho \phi_\tau(\bar \varrho^\tau)\rangle \bar \mu^\tau_t(dx, dv)\rangle=0 
\end{equation} 
The uniform bound in (\ref{eq:jan30.2017.2}) implies that 
\[
\sup_{t, \tau} \int_{TM} |v|^2 \bar \mu_t^\tau(dx, dv)<\infty..
\]
Thus, if $A \in C_c^\infty((0,T) \times M)$, $B_i(v)\equiv v_i$, since $B_i$ grows slower than $|v|^2$ at infinity, by a standard approximation argument, we can use $L(t,x,v):= B_i(v) A(t,x)$ in (\ref{eq:distribution1}) and read off the proof of (i).

(ii) Applying Remark \ref{rem:error-fisher} (iii), we obtain in the sense of distributions 
\[
 \nabla \Bigl( {1\over 2} \nabla (\triangle \varrho^\tau) - {\rm div} \bigl( \nabla \sqrt \varrho^\tau \otimes  \nabla \sqrt \varrho^\tau \bigr) \Bigr) + \varrho^\tau \xi^\tau=0.
\]
This, together with (i), imply (ii).

(iii) For any $A \in C_c^\infty(M)$ such that $||\nabla A||_\infty \leq 1$, we have 
\[
\int_{M}  \langle \nabla_\varrho  \phi_\tau(\bar \varrho_t^\tau); A\rangle  \bar \varrho_t^\tau(dx)= 
\int_{M}  \langle \nabla_\varrho  \phi_\tau(\bar \varrho_t^\tau)\circ  {\bf t}^\tau; A( {\bf t}^\tau) \rangle  \varrho_t^\tau(x) dx
\]
Thus, using (\ref{eq:gradient-fisher}) we conclude that 
\[
\Bigl|\int_{M} \langle\vec 0^\tau ; A\rangle  dx \Bigr| = \Bigl| \int_{M}   \Bigl \langle  \xi^\tau ; A ( {\bf t}^\tau)- A(\id )  \Bigr\rangle \varrho_t^\tau(x) dx\Bigr| \leq 
 ||\xi^\tau||_{\varrho_t^\tau} ||{\bf t}^\tau -\id||_{\varrho_t^\tau}=    ||\xi^\tau||_{\varrho_t^\tau} W_2(\varrho_t^\tau, \bar \varrho_t^\tau)
\]
Since by Remark \ref{re:moreau-sub} (ii) 
\[
W_2(\varrho_t^\tau, \bar \varrho_t^\tau) ||\xi^\tau||_{\varrho_t^\tau} \leq {W_2^2(\varrho_t^\tau, \bar \varrho_t^\tau) \over \tau} = 2\bigl(\phi_{\tau}(\bar \varrho_t^\tau) - \phi(\varrho_t^\tau)\bigr), 
\]
we obtain  
\[
\Bigl|\int_{M} \langle\vec 0^\tau_t ; A\rangle  dx \Bigr|  \leq  2 \bigl(\phi_{\tau}(\bar \varrho_t^\tau) - \phi(\varrho_t^\tau)\bigr).
\]
Hence, 
\begin{equation}\label{eq:estimate-kin1}
|| \vec 0^\tau_t ||_{-1}  \leq  2 \bigl(\phi_{\tau}(\bar \varrho_t^\tau) - \phi(\varrho_t^\tau)\bigr).
\end{equation}  
We use the fact that $\phi \geq 0$ and (\ref{eq:feb23.2017.2})  to obtain for any $t \in (0,T)$ and $\tau \in (0,1)$   
\[
\phi_{\tau}(\bar \varrho_t^\tau) - \phi(\varrho_t^\tau) \leq \phi_{\tau}(\bar \varrho_t^\tau) \leq \mathcal H(\mu_0)+||V||_\infty.
\] 
We can use (\ref{eq:lebesgued}) and the Lebesgue dominated convergence theorem to conclude that for any $p \geq 1$ 
\[
\lim_{n \rightarrow \infty} \int_0^T || \vec 0^{\tau_n}_t||_{-1}^p dt =0. 
\]\endproof

%
%
\section{Concluding remarks} 
It is important to mention that the previous results require the initial condition to be absolutely continuous with respect to the Lebesgue measure, and therefore, the mono-kinetic case presented in the introduction is not covered. It remains an interesting question to determine if our results may be extended to an arbitrary initial measure if we consider the second method proposed in \cite{ambrosio2008hamiltonian}. 

On the other hand, the convergence analysis needs to be improved in order to verify that the limit of the approximative scheme satisfies the kinetic Bohmian equation in a weak sense. We leave it as an open question for now to investigate if the flow exchange technique introduced in \cite{matthes2009family} for the analysis of Wasserstein gradient flows may be extended to our problem, giving us the additional estimates that we need to pass to the limit in our approximative scheme.

%
%
\section*{Acknowledgments.} The research of W. Gangbo was supported by NSF grant  DMS--1160939.

%
%

\end{document}